\documentclass[oneside,12pt]{amsart}
\usepackage[parfill]{parskip}

\usepackage{tikz-cd}
\usepackage{amsmath}
\usepackage{geometry}
\usepackage{latexsym}
\usepackage{amssymb}
\usepackage{amscd}
\usepackage{bm}
\usepackage[font=small,labelfont=bf]{caption}
\usepackage{graphicx}
\usepackage{epsfig}
\usepackage{epstopdf}
\usepackage{placeins}
\usepackage{mathtools}
\usepackage{subfigure}
\usepackage{stackengine}

\DeclarePairedDelimiter{\ceil}{\lceil}{\rceil}
\DeclarePairedDelimiter{\floor}{\lfloor}{\rfloor}
\providecommand{\N}{\mathbb{N}}
\providecommand{\R}{\mathbb{R}}

\providecommand{\Z}{\mathbb{Z}}

\providecommand{\abs}[1]{\left\vert#1\right\vert}

\providecommand{\set}[1]{\left\{#1\right\}}

\providecommand{\floor}[1]{\lfloor\{#1\rfloor\}}

\providecommand{\paren}[1]{\left( #1 \right)}

\providecommand{\brac}[1]{\left [ #1 \right]}
\providecommand{\set}[1]{\left { #1 \right}}

\newcommand{\s}[1]{\begin{equation*} \begin{split} #1 \end{split} \end{equation*}}

\providecommand{\eqref}[1]{\left(\ref{#1}\right)}

\newtheorem{theorem}{Theorem}
\newtheorem{Lemma}[theorem]{Lemma}
\newtheorem{conj}[theorem]{Conjecture}

\newtheorem{Definition}[theorem]{Definition}
\newtheorem{Proposition}[theorem]{Proposition}
\newtheorem{Question}[theorem]{Question}
\newtheorem{cor}[theorem]{Corollary}

\newcommand{\PH}{\mathit{PH}}
\begin{document}
\title[Persistent Homology and the Upper Box Dimension]{Persistent Homology and \\ the Upper Box Dimension}
\author{Benjamin Schweinhart}
\date{July 2019}

\begin{abstract}
We introduce a fractal dimension for a metric space defined in terms of the persistent homology of extremal subsets of that space. We exhibit hypotheses under which this dimension is comparable to the upper box dimension; in particular, the dimensions coincide for subsets of $\R^2$ whose upper box dimension exceeds $1.5.$ These results are related to extremal questions about the number of persistent homology intervals of a set of $n$ points in a metric space. 
\end{abstract} 
\maketitle

\section{Introduction}

Several notions of fractal dimensions based on persistent homology have been proposed in the literature, including in the PhD thesis of Vanessa Robins~\cite{2000robins}, in a paper written by Robert MacPherson and the current author~\cite{2012macpherson}, and by Adams et al.~\cite{2019adams}. In that work, empirical estimates of the proposed dimensions were compared with classically defined fractal dimensions. Here, we prove the first rigorous analogue of those comparisons.

\begin{figure}
\centering  
\subfigure[]{
\label{fig:sierpinski}
\includegraphics[height=0.23\linewidth]{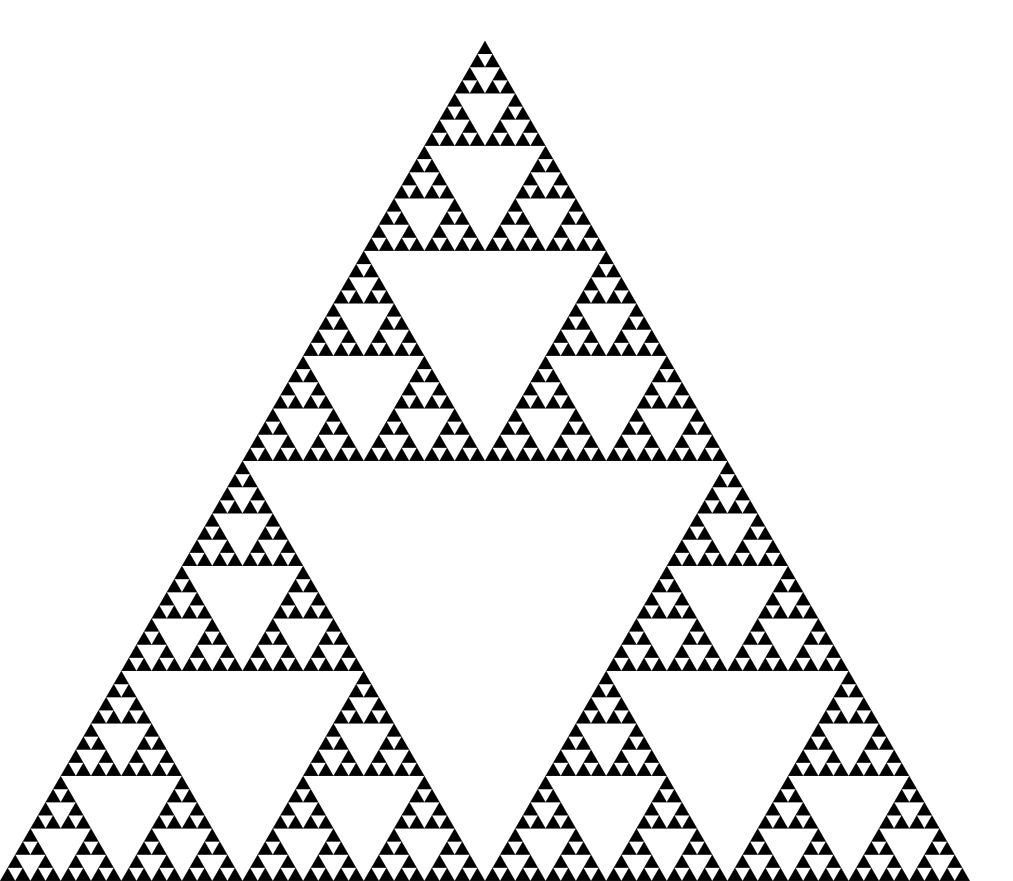}
}
\hspace{.05\linewidth}
\subfigure[]{
	\label{fig:percolation}
	\includegraphics[height=0.23\linewidth]{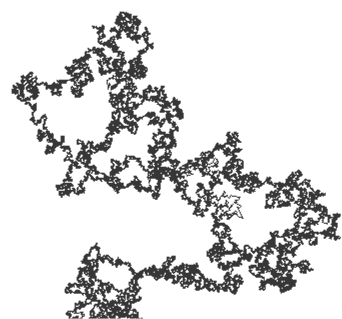}}
\hspace{.05\linewidth}
\subfigure[]{
	\label{fig:ikeda}
	\includegraphics[height=0.23\linewidth]{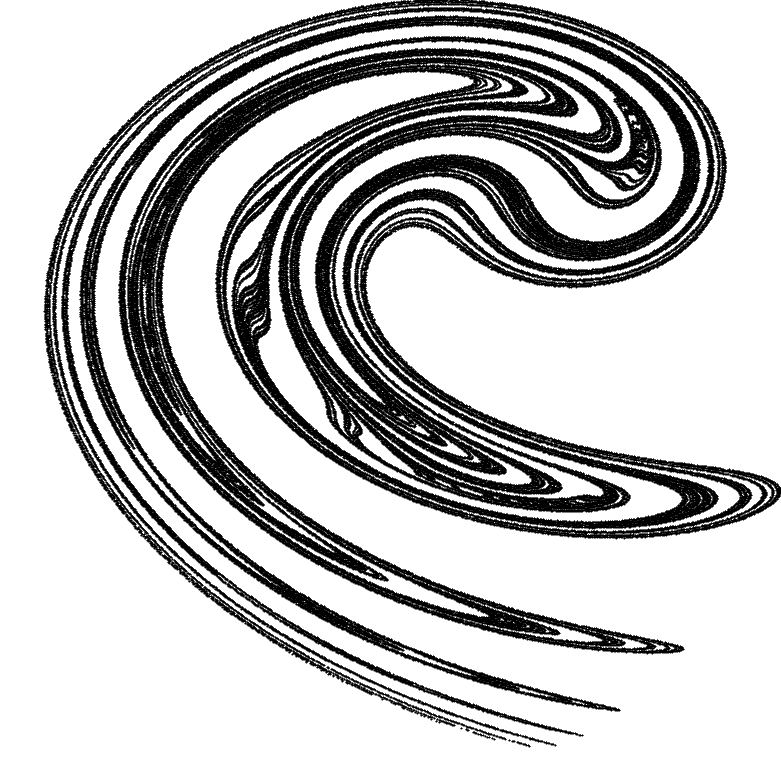}}
	\caption{\label{fig:examples}Examples which satisfy the hypotheses of our main theorem: the Sierpinski triangle ($\text{dim}_{\text{box}}=\log\paren{3}/\log\paren{2}\approx 1.585$), Schramm--Loewner Evolution with $\kappa=6$ ($\text{dim}_{\text{box}}\geq \text{dim}_{\text{Hausdorff}}=7/4$~\cite{2004beffara}), Figure provided by T. Kennedy~\cite{2009kennedy}), and the Ikeda attractor ($\text{dim}_{\text{box}}\approx 1.7$ assuming that computational estimates of the box dimension of that set accurate).}
\end{figure}

The motivation for the definition proposed here comes from the literature on minimal spanning trees. The properties of minimal spanning trees of point collections contained in a bounded subset of $\R^m$ have long been of interest~\cite{1988steele,1996kesten}. In 2005, Kozma, Lotker, and Stupp~\cite{Kozma} proved a relationship between the extremal behavior of these trees and the upper box dimension. To be precise, let $T\left(\textbf{y}\right)$ denote the  minimal spanning tree of a finite metric space $\textbf{y}$   and let
\[E_{\alpha}^0\left(\textbf{y}\right)=\frac{1}{2}\sum_{e \in T\left(Y\right)}\abs{e}^{\alpha}\,,\]
where the sum is taken over all edges $e$ in the tree $T\left(\textbf{y}\right),$ and  $\abs{e}$ denotes the length of the edge.
\begin{Definition}
 If $X$ is a bounded metric space, let $\text{dim}_{\text{MST}}\paren{X}$ be the infimal exponent $\alpha$ so that $E_\alpha^0\paren{\textbf{x}}$ is uniformly bounded for all finite point sets $\textbf{x}\subset X:$

\[\text{dim}_{\text{MST}}\paren{X}=\inf\set{\alpha : \exists\,C\text{ so that } E_\alpha^0\paren{\textbf{x}} < C\;\forall \text { finite }\;\textbf{x}\subset X}\,.\]
\end{Definition}

Kozma, Lotker, and Stupp proved that $\text{dim}_{\text{MST}}\paren{X}$ equals the upper box dimension~\cite{1928bouligand} of $X.$
\begin{Definition}
\label{defn_upperbox}
Let $X$ be a bounded metric space and let $N_\delta\paren{X}$ be the maximal number of disjoint closed $\delta$-balls with centers in $X.$ The \textbf{upper box dimension} of $X$ is
\s{\text{dim}_{\text{box}}\paren{X}=\limsup_{\delta\rightarrow 0}{\frac{\log\paren{N_\delta\paren{X}}}{\log\paren{1/\delta}}}\,.}
\end{Definition}

\begin{theorem}[Kozma, Lotker, and Stupp~\cite{Kozma}]
For any metric space $X,$
\[\text{dim}_{\text{MST}}\paren{X}=\text{dim}_{\text{box}}\paren{X}\,.\]
\end{theorem}

If $\textbf{x}$ is a finite point set contained in a bounded metric space, and $\PH_i\left(\textbf{x}\right)$ is the $i$-dimensional persistent homology of the \v{C}ech complex of $\textbf{x},$ there is a bijection between the edges of the Euclidean minimal spanning tree of $\textbf{x}$ and the intervals in the canonical decomposition of $\PH_0\left(\textbf{x}\right),$ where the length of an interval in $\PH_0\left(\textbf{x}\right)$ is half the length of the corresponding edge. (Note that if persistent homology is taken of the Rips complex of $\textbf{x},$ then a similar correspondence holds without the requirement of an ambient space, where an interval corresponds to the edge of the same length.) This observation suggests a generalization of the previous result to higher-dimensional persistent homology. Namely, let
\[E_{\alpha}^i\left(\textbf{x}\right)=\sum_{\left(b,d\right)\in \PH_i\left(\textbf{x}\right)} \paren{d-b}^{\alpha}\]
where the sum is taken over all bounded $\PH_i$ intervals and define:
\begin{Definition}
\label{defn_main}
Let $X$ be a bounded subset of a metric space. The \textbf{$\PH_i$-dimension} of $X$ is
\s{\text{dim}_{\PH}^i\paren{X}=\inf\set{\alpha : \exists\,C\text{ so that } E_\alpha^i\paren{\textbf{x}} < C\;\forall \text{ finite }\textbf{x}\subset X}\,.}
Also, let  $\text{dim}_{\widetilde{\PH}}^i\paren{X}$ be defined by replacing the \v{C}ech complex with the Rips complex in the previous construction. 
\end{Definition}
Then 
\[\text{dim}_{\PH}^0\paren{X}=\text{dim}_{\widetilde{\PH}}^0\paren{X}=\text{dim}_{\text{MST}}\paren{X}\,.\]
Note that $\text{dim}_{\widetilde{\PH}}^i\paren{X}$ is defined for bounded metric spaces $X,$ rather than bounded subsets of a metric space.

\begin{Question}
\label{MainQuestion}
Are there hypotheses on $X$ under which $\text{dim}_{\PH}^i\paren{X}=\text{dim}_{\text{box}}\paren{X}$ or  $\text{dim}_{\widetilde{\PH}}^i\paren{X}=\text{dim}_{\text{box}}\paren{X}$ for $i>0$?
\end{Question}
Unlike the 0-dimensional case, equality will not always hold. For example, the $\PH_1$-dimension of a line in $\R^m$ will always be $0.$ Also, there are metric spaces not embeddable in any finite-dimensional Euclidean space whose $\widetilde{\PH}_1$-dimension exceeds their upper box dimension.

\begin{Proposition}
\label{prop:ripsExample} There is a metric space $X$ so that $\text{dim}_{\widetilde{\PH}_1}\paren{X}=2$ but $\text{dim}_{\text{box}}\paren{X}=1.$ \end{Proposition}
We prove this in Section~\ref{RipsExample}. This behavior is related to the existence of families of point sets $\set{\textbf{x}_n}_{n\in\N}$ for which number of intervals of $\PH_i\paren{\set{\textbf{x}_n}}$ grows slower or faster than linearly in in $\abs{\textbf{x}_n}$ when $i>0$; by contrast, the reduced $0$-dimensional persistent homology of a finite metric space with $k$ points always has $k-1$ intervals.

We restrict our attention to subsets of Euclidean space, and conjecture

\begin{conj}
\label{main_conjecture}
For any $m\in\N$ and $0\leq i < m,$ there is a constant $\gamma_i^m< m$ so that if $X\subset\mathbb{R}^m$ and $\text{dim}_{\text{box}}\left(X\right)>\gamma_i^m$ then
\s{\text{dim}_{\PH}^i\paren{X}= \text{dim}_{\text{box}}\paren{X}\,.}
\end{conj}

Cohen-Steiner, Edelsbrunner, Harer, and Mileyko studied a quantity similar to $E_{\alpha}^i$ in their paper~\cite{2010CohenSteiner}. Their results immediately imply that if $X\subset\R^m,$ then $\dim_{\PH}^i\paren{X}\leq m$ for all $i\in \N$ (Corollary~\ref{corollary_upper_2}). If $X$ is a subset of $\R^m$ with non-empty interior the corresponding lower bound is easy to show and $\dim_{\PH}^i\paren{X}=m$ for $i=1,\ldots,m-1$ ( Proposition~\ref{prop:interior}). However, our focus here will be to prove results about subsets of fractional box dimension. This task is challenging, and involves difficult combinatorial problems. Our main result is:
 
\begin{theorem}
\label{mainTheorem}
Let $X$ be a bounded subset of $\R^2.$ If $\text{dim}_{\text{box}}\left(X\right)>1.5,$ then 

\s{\text{dim}_{\PH}^1\paren{X}= \text{dim}_{\text{box}}\paren{X}\,.}
\end{theorem}
See Figure~\ref{fig:examples} for images of two examples known to meet these hypotheses, and one that is believed to based on computational experiments. The upper bound $\text{dim}_{\PH}^1\paren{X} \leq \text{dim}_{\text{box}}\paren{X}$ is proven in Section~\ref{sec_upper_bound}, and the lower bound in Section~\ref{sec_lower_bound}.

We also prove partial results in more general cases:
\begin{theorem}
\label{thmLower}
Let $X$ be a bounded subset of $\R^m.$ If $\text{dim}_{\text{box}}\left(X\right)>m-1/2,$ then 
\s{\text{dim}_{\text{box}}\paren{X}\leq \text{dim}_{\PH}^1\paren{X}\leq m\,.}
\end{theorem}

Furthermore, we show that the example in Proposition~\ref{prop:ripsExample} cannot be taken to be a subset of $\mathbb{R}^m.$
\begin{theorem}
\label{ripsUpper}
\label{ripsTheorem}
If $X$ is a bounded subset of $\R^m$ then 
\s{\text{dim}_{\widetilde{\PH}}^1\paren{X}\leq\text{dim}_{\text{box}}\paren{X}\,.}
\end{theorem}

In the process of proving this, we also show:
\begin{theorem}
\label{theorem_rips_count}
If $\textbf{x}$ is a finite subset of $\R^m$ then the first-dimensional persistent homology of the Rips complex of $\textbf{x}$ contains $O\paren{\abs{\textbf{x}}}$ intervals.
\end{theorem}

In previous work with MacPherson~\cite{2012macpherson}, we defined an alternate notion of persistent homology dimension that measures the complexity of a shape rather than a classical notion of fractal dimension. In Appendix~\ref{sec:compWithMeas}, we show that it is a lower bound for $\text{dim}_{\PH}^1.$

We prove the upper bounds for Theorems~\ref{mainTheorem} and~\ref{thmLower} in Section~\ref{sec_upper_bound}, and the lower bounds in Section~\ref{sec_lower_bound}. The results and arguments will be stated for the \v{C}ech complex, but we will indicate which ones also work for the Rips complex. We prove our results specific to the Rips complex in Section~\ref{sec:rips}. First, we cover some preliminaries.

\section{Preliminaries}

In the introduction we defined the upper box dimension in terms of $N_\delta\paren{X},$  maximal number of disjoint closed balls centered at points of a bounded metric space $X:$
\s{\text{dim}_{\text{box}}\paren{X}=\limsup_{\delta\rightarrow 0}{\frac{\log\paren{N_\delta\paren{X}}}{\log\paren{1/\delta}}}\,.}
Note that this is equivalent to taking $\text{dim}_{\text{box}}\paren{X}$ to be the unique real number $d$ so that
\begin{itemize}
\item $N_\delta\paren{X}=O\paren{\delta^{-\alpha}}$ for all $\alpha>d.$ That is, for all $\alpha>d$ there exists a $C>0$ so that $N_\delta\paren{X}<C\delta^{-\alpha}$ for all $\delta>0.$
\item $N_\delta\paren{X}\neq O\paren{\delta^{-\alpha}}$ for all $\alpha<d.$ That is, for all $\alpha<d$ and all $D>0$ there exists a sequence $\delta_j\rightarrow 0$ so that $N_\delta\paren{X}>D\delta^{-\alpha}$ for all $j.$
\end{itemize}
There are other equivalent definitions. In particular $N_{\delta\paren{X}}$ can be replaced by the minimum number of balls of radius $\delta$ required to cover $X$ or, if $X\subset\mathbb{R}^m,$ the number of cubes in a regular square tesselation of $\mathbb{R}^m$ of width $\delta$ which intersect $X.$ See Falconer~\cite{Falconer} for details.

Also, if $\Delta$ is an $k$-dimensional simplex in $\mathbb{R}^m$, the \textbf{circumsphere} of $\Delta$ is the smallest $m-1$-dimesional sphere containing the vertices of $\Delta.$ The \textbf{circumradius} of $\Delta$ is the radius of its circumsphere.

In the following, bold lower case letters will denote finite point sets  (i.e. $\textbf{x},\textbf{y}$) .

\subsection{Persistent Homology}
Here, we provide a brief introduction to persistent homology. For a more thorough exposition, see see~\cite{2008edelsbrunner,2010harer,2016chazal}. We assume the reader is familiar with the basics of simplicial homology, as in~\cite{2002hatcher}. 

\subsubsection{Filtrations}
\begin{figure}
\centering
\includegraphics[width=.7\textwidth]{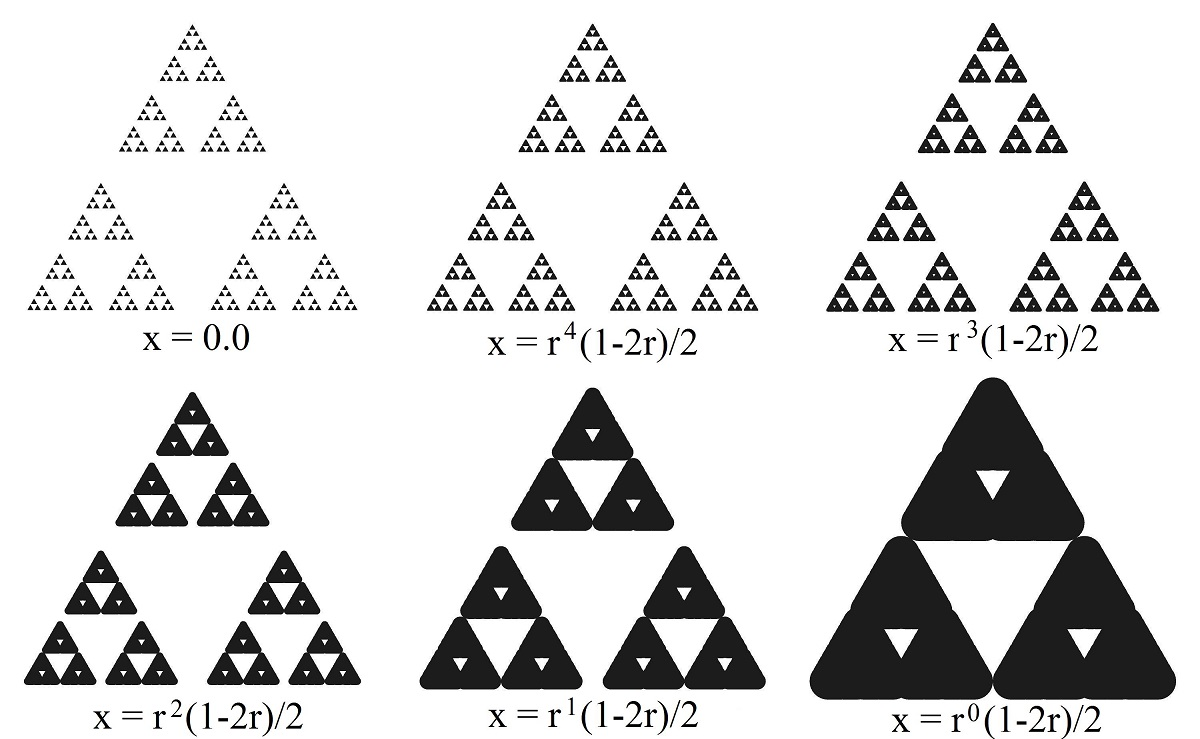}
\caption{A self-similar fractal and its $r$-neighborhoods~\cite{2012macpherson}.}
\label{fig:triangles}
\end{figure}

A \textbf{filtration} is a sequence of topological spaces $\set{S_\alpha}_{\alpha\in I}$ together with inclusion maps $i:S_\alpha\rightarrow S_\beta$ for $\alpha,\beta \in I, \alpha<\beta,$ where $I$ is an ordered index set --- usually the positive real numbers, the natural numbers, or a finite set of real numbers. For example, if $X$ is a subset of a metric space $M$ and $\epsilon>0$ the $\epsilon$-neighborhood filtration of $X$ is $\set{X_{\epsilon}}_{\epsilon\in\mathbb{R}^+}$ where
\[X_{\epsilon}=\set{x\in M: d\paren{x,y}<\epsilon\text{ for some }y\in X}\,.\]
See Figure~\ref{fig:triangles} for an example. 

We define three commonly used filtrations, each indexed by a positive real number $\epsilon.$ If $X$ is a subset of a metric space, the \textbf{\v{C}ech complex} on $X$ is the filtration of simplicial complexes $C_{\epsilon}\paren{X},$ defined by
\[ \paren{x_1,\ldots, x_k}\in C_{\epsilon}\paren{X}\text{ if } \cap_{j=1}^k B_{\epsilon}\paren{x_{j}}\neq \varnothing\,,\]
where $x_1,\ldots x_k$ are points in $X.$ For example if $X$ is a subset of Euclidean space, the $2$-simplex formed by the vertices of an acute triangle enters the \v{C}ech complex when $\epsilon$ is the circumradius of that triangle. Also in the Euclidean case, $C_{\epsilon}\paren{X}$ is homotopy equivalent to the $\epsilon$-neighborhood $X_{\epsilon}.$

If $\textbf{x}$ is a finite subset of Euclidean space, we can define another filtration called the Alpha complex on $\textbf{x},$ $A_{\epsilon}\paren{\textbf{x}}.$ It is smaller than the \v{C}ech complex on $\textbf{x},$ but contains equivalent topological information: $A_{\epsilon}\paren{\textbf{x}}$ is homotopy equivalent to $C_{\epsilon}\paren{\textbf{x}}$ for all $\epsilon>0.$~\cite{2002edelsbrunner} We define the Alpha complex for point sets in general position; a more general definition is in~\cite{2002edelsbrunner}. First, we require the definition of the Delaunay triangulation~\cite{1934delaunay,2012cheng}.

Let $\textbf{x}$ be a finite subset of $\mathbb{R}^m$ in general position. The \textbf{Delaunay triangulation} $\text{DT}\paren{\textbf{x}}$ on $\textbf{x}$ is the unique triangulation of $\mathbb{R}^m$ with the property that if $\Delta$ is any $m$-simplex in $DT\paren{\textbf{x}}$ and $S\paren{\Delta}$ is its circumsphere, then no point of $\textbf{x}$ is contained in the bounded component of $S\paren{\Delta}.$ 

Now, we can define the Alpha complex of a point set $\textbf{x}\subset\mathbb{R}^m$ in general position. The \textbf{Alpha complex} $A_{\epsilon}\paren{\textbf{x}}$ is a filtration of subcomplexes of the Delaunay triangulation on $\textbf{x},$ $\text{DT}\paren{\textbf{x}}.$ A simplex $\sigma\in \text{DT}\paren{\textbf{x}}$ is contained in $A_{\epsilon}\paren{\textbf{x}}$ if either of the following two conditions is met:
\begin{itemize}
\item The circumradius of $\sigma$ is less than or equal to $\epsilon,$ and no points of $\textbf{x}$ are contained in the bounded component of the complement of its circumsphere.
\item $\sigma$ is a simplex of a higher-dimensional simplex contained in $A_{\epsilon}\paren{\textbf{x}}.$
\end{itemize}

Finally, we define the Rips complex~\cite{1927vietoris,2007de_silva} of a metric space which, unlike the two previous constructions, does not depend on an ambient space.
If $X$ is a metric space, the \textbf{Rips complex} $R_{\epsilon}\paren{X}$ is the simplicial complex with vertex set $X$ so that
\[\paren{x_{1},\ldots,x_{l}}\in R_{\epsilon}\paren{X}\text{ if }d\paren{x_{j},x_{k}}\leq \epsilon\text{ for } \,1\leq j < k \leq l\,,\]
where $x_1,\ldots,x_l$ are points in $X.$

\subsubsection{Definition of Persistent Homology}

Let $H_i\paren{X}$ denote the reduced homology of a topological space, with coefficients in a field $k$.\footnote{We assume the reader is familiar with the basics of simplicial homology, as in~\cite{2002hatcher}.} If $X_{\alpha}$ is a filtration of topological spaces, the \textbf{persistence module} of $X_{\alpha},$ is the product $\prod_{\alpha} H_i\paren{X_\alpha},$ together with the maps $i_{\alpha,\beta}:H_i\paren{X_{\alpha}}\rightarrow H_i\paren{X_{\beta}}$ for $\alpha<\beta.$ If the rank of $i_{\alpha,\beta}$ is finite for all $\alpha<\beta,$ the structure of the persistent homology of $X_\alpha$ is captured by a unique set of intervals~\cite{2005zomorodian,2009chazal} $\PH_i\paren{X}$ so that
\[\text{rank}\:i_{\alpha,\beta}=\abs{\set{I\in \PH_i\paren{X}:\brac{\alpha,\beta}\subseteq I}}\,.\]
If $Y$ is a compact  subset of a metric space, the Rips and \v{C}ech filtrations on $X$ satisfy the required finiteness hypotheses~\cite{2014chazal,2016chazal}. In the following $\PH_i\paren{X}$ will refer to the set of intervals of the $i$-dimensional persistent homology of the \v{C}ech filtration of $X$ and $\widetilde{\PH}_i\paren{X}$ will refer to the $i$-dimensional persistent homology of the Rips filtration of $X.$ Note that $\PH_i\paren{X}$ depends on the ambient space.

We will repeatedly use a stability theorem for persistent homology~\cite{2007cohen,2014chazal}. If $X$ and $Y$ are filtrations, let the bottleneck distance between $\PH_i\paren{X}=\set{\paren{b_i,d_i}}$ and $\PH_i\paren{Y}=\set{\paren{\hat{b}_j,\hat{d}_j}}$ be 
\[d_B\paren{\PH_i\paren{X},\PH_i\paren{Y}}=\inf_{\eta}\sup_{j}\max\paren{\abs{b_j-\hat{b}_{\eta\paren{j}}},\abs{d_j-\hat{d}_{\eta\paren{j}}}}\]
where $\eta$ ranges over all partial matchings between the intervals $\PH_i\paren{X}$ and $\PH_i\paren{Y},$ allowing intervals from either set to be matched to intervals along the diagonal (that is, ones of the form $(d,d)$). Also, let $d_H$ denote the Hausdorff distance between subsets of a metric space
\[d_H\paren{X,Y}=\text{max}\paren{\sup_{x\in X}\inf_{y\in Y}d\paren{x,y},\sup_{y\in Y}\inf_{x\in X}d\paren{x,y}}\,.\]

\begin{theorem}(Stability of the Bottleneck Distance~\cite{2007cohen})
\label{thm_stable}
If $X$ and $Y$ are bounded subsets of a metric space and $i\in \N$ then
\[d_B\paren{\PH_i\paren{X},\PH_i\paren{Y}}\leq d_H\paren{X,Y}\,.\]
\end{theorem}

In particular, if $I_{i,\epsilon}\paren{X}$ is the set of $\PH_i$ intervals of $X$ of length greater than $\epsilon$
\begin{equation}
\label{eqn_Iie}
I_{i,\epsilon}\paren{\textbf{x}}=\set{\paren{b,d}:\PH_i\paren{\textbf{x}}:d-b>\epsilon}\,,
\end{equation}
then we have the following result.
\begin{cor}
\label{stableCorollary}
Let $X$ and $Y$ be compact subsets of a metric space and $\epsilon,\delta>0.$ If $d_H\paren{X,Y}<\delta/2$ then
\[\abs{I_{i,\epsilon+\delta}\paren{X}}\leq \abs{I_{i,\epsilon}\paren{X}}\,.\]
\end{cor}

The preceding theorem and corollary are also true for $\widetilde{PH}_i.$ 

\subsection{Persistent Homology Dimension} 

Recall that the $\PH_i$ dimension of a bounded subset of a metric space is
\s{\text{dim}_{\PH}^i\paren{X}=\inf\set{\alpha : \exists\,C\text{ so that } E_\alpha^i\paren{\textbf{x}} < C\;\forall \text{ finite }\textbf{x}\subset X}\,.}
A straightforward argument based on stability of the bottleneck distance shows that we could modify the definition of $\text{dim}_{\PH}^i(X)$ to consider all compact subsets of $X,$  rather than all finite subsets: 
\s{\text{dim}_{\PH}^i\paren{X}=\inf\set{\alpha : \exists\,C\text{ so that }E_\alpha^i\paren{Y} < C\;\forall \text{ compact } \;Y\subseteq X}\,.}
In Proposition~\ref{prop_upper1} below, we provide another equivalent definition of the $\PH_i$ dimension, in terms of the asymptotic number of ``long'' intervals of ``well-spaced'' point sets of $X.$

Also, for the Rips complex define
\s{\widetilde{E}_\alpha^i\paren{\textbf{x}}=\sum_{\paren{b,d}\in \widetilde{\PH}_i\paren{\textbf{x}}}\paren{d-b}^{\alpha}\,,}
where the sum is taken over the finite intervals of $\widetilde{\PH}_i\paren{\textbf{x}}.$ With that,
\s{\text{dim}_{\widetilde{\PH}}^i\paren{X}=\inf\set{\alpha :\exists\,C\text{ so that } \widetilde{E}_\alpha^i\paren{\textbf{x}} < C\;\forall\text{ finite }\textbf{x}\subset X}\,.}

\section{An Upper Bound}
\label{sec_upper_bound}

 Our strategy for bounding  $\text{dim}_{\PH}^i\paren{X}$ in terms of $\text{dim}_{\text{box}}\paren{X}$ is to approximate subsets of $X$ by finite point sets whose size is controlled by the box dimension, and apply bottleneck stability. In this section, Lemmas~\ref{lemma_upper_new1} and~\ref{lemma_upper_new} and Propositions~\ref{prop_upper1} and~\ref{prop1} hold for the Rips complex using identical arguments, but Corollaries~\ref{corollary_upper_2} and~\ref{corollary_upper} only apply to the \v{C}ech complex.

\begin{Lemma}
\label{lemma_upper_new1}
Let $X$ be a bounded metric space and suppose there are positive real numbers $c$ and $D_0$ so that for all $\epsilon>0$ and all $\textbf{x}\subset X$
\begin{equation}
\label{eqn_h0}
\abs{I_{i,\epsilon}\paren{\textbf{x}}}<D_0\epsilon^{-c}\,,
\end{equation}
where $I_{i,\epsilon}\paren{\textbf{x}}$ is defined in Equation~\ref{eqn_Iie}.
Then
$\text{dim}_{\PH}^i\paren{X}\leq c.$
\end{Lemma}
\begin{proof}
Rescale $X$ if necessary so its diameter is less than one. Note that this implies that $\abs{I}\leq 1$ for all $I\in\PH_i\paren{X}.$ 

Let $\alpha>c$ and $\textbf{x}\subseteq X.$ We will bound $E_i^\alpha\paren{\textbf{x}}$ by summing over the contributions of the intervals whose lengths are between $2^{-k-1}$ and $2^{-k}.$  For $k\in\mathbb{N}$ let
\begin{equation}
\label{eqn_h1}
J_{i,k}\paren{\textbf{x}}=\set{I\in \PH_i\paren{\textbf{x}}:2^{-k-1} < \abs{I} \leq  2^{-k}}\,.
\end{equation}

Then
\begin{align*}
E_\alpha^i\paren{\textbf{x}}=& \;\; \sum_{k=0}^\infty\sum_{I\in J_{i,k}\paren{\textbf{x}}}\abs{I}^\alpha \\
\leq & \;\; \sum_{k=0}^\infty\abs{J_{i,k}\paren{\textbf{x}}}2^{-\alpha k} &&\text{by Eqn.~\ref{eqn_h1}}\\
\leq & \;\; \sum_{k=0}^\infty\abs{I_{i,2^{-k-1}}\paren{\textbf{x}}}2^{-\alpha k} &&\text{by Eqn.~\ref{eqn_Iie}}\\
\leq &  \;\; \sum_{k=0}^\infty D_0 2^{kc+c}2^{-\alpha k}&& \text{by Eqn.~\ref{eqn_h0}}\\
=& \;\; D_0 2^{c} \sum_{k=0}^\infty\paren{2^{c-\alpha}}^k\\
= & \;\; D_0 2^c \frac{1}{1-2^{c-\alpha}} && \text{because $\alpha>c$\,.}
\end{align*}
Therefore, $E_\alpha^i\paren{\textbf{x}}$ is uniformly bounded for all $\textbf{x}\subset X,$ $\text{dim}_{\PH}^i\paren{X}\leq \alpha$ for all $\alpha>c,$ and $\text{dim}_{\PH}^i\paren{X}\leq c.$ 
\end{proof}

As noted by Cohen-Steiner et. al.~\cite{2010CohenSteiner}, if $\textbf{x}$ is a subset of a triangulable metric space $M$ then $\abs{I_{i,\epsilon}\paren{\textbf{x}}}$ is less than the number of simplices in a triangulation of mesh $\epsilon$ of $M$. In particular, if $\textbf{x}$ is a finite subset of $\mathbb{R}^m$ (or an $m$-dimensional Riemannian manifold) then $\abs{I_{i,\epsilon}\paren{\textbf{x}}}=O\paren{\epsilon^{-m}}.$ 

\begin{cor}(Cohen-Steiner, Edelsbrunner, Harer, and Mileyko~\cite{2010CohenSteiner})
\label{corollary_upper_2}
Let $X$ be a bounded subset of $\R^m.$ Then
\[\text{dim}_{\PH}^i\paren{X}\leq m\,.\]
\end{cor}
Which is the upper bound in Theorem~\ref{thmLower}.

In the following $X$ will be a bounded metric space, and $\textbf{x}^\epsilon$ will be the centers of a maximal collection of disjoint balls of radius $\epsilon/4$ centered at points of $X.$ By the maximality of $\textbf{x}^\epsilon,$ the balls of radius $\epsilon/2$ centered at the points of $\textbf{x}^\epsilon$ cover $X$ and
\begin{equation}
d_H\paren{\textbf{x}^{\epsilon},X}<\epsilon/2\,.
\end{equation}

\begin{Lemma}
\label{lemma_upper_new}
Let $X$ be a bounded metric space and let $c, D>0.$ For $\epsilon>0$ let $\textbf{x}^\epsilon$ be the centers of a maximal collection of disjoint balls of radius $\epsilon/4$ centered at points of $X.$ If 
\begin{equation}
\label{eqn_h2}
\abs{I_{i,\epsilon}\paren{\textbf{y}}}<D\epsilon^{-c}\text{ for all }\textbf{y}\subseteq \textbf{x}^\epsilon\,,
\end{equation}
for all $\epsilon>0$ then
\[\text{dim}_{\PH}^i\paren{X}\leq c\,.\]
\end{Lemma}
\begin{proof}
Rescale the metric if necessary so that the diameter of $X$ is less than one, and let $\alpha>c.$ We will show that $E^i_{\alpha}\paren{\textbf{y}}$ is uniformly bounded for all $\textbf{y}\subset X.$ 

Let $\textbf{y}\subset X,$ $\epsilon>0,$ and  
\[\textbf{y}^{\epsilon}=\set{x\in \textbf{x}^{\epsilon}: d\paren{x,\textbf{y}}<\epsilon/2}\,.\]
Every point of $\textbf{y}^{\epsilon}$ is within distance $\epsilon/2$ of a point of $\textbf{y}$, and every point of $\textbf{y}$ is within distance $\epsilon/2$ of a point of $\textbf{y}^\epsilon$ because $d_H\paren{\textbf{x}^{\epsilon},X}<\epsilon/2.$ It follows that $d_H\paren{\textbf{y}^\epsilon,\textbf{y}}<\epsilon/2.$ 

By stability (Corollary~\ref{stableCorollary}),
\[\abs{I_{i,2\epsilon}\paren{\textbf{y}}}\leq \abs{I_{i,\epsilon}\paren{\textbf{y}^\epsilon}} \leq D \epsilon^{-c}\,,\]
for all $\epsilon>0.$ 
Then
\[\abs{I_{i,\epsilon}\paren{\textbf{y}}} \leq 2^{c} D \epsilon^{-c}\,,\]
and the desired result follows from Lemma~\ref{lemma_upper_new1}.
\end{proof}

Next, we show that we can characterize $\text{dim}_{\PH}^i\paren{X}$ in terms of the number of ``long'' intervals of ``well-spaced'' point sets in $X.$
If $X$ is a bounded metric space and $\textbf{x}^\epsilon$ is a finite point set consisting of the centers of a maximal collection of balls of radius $\epsilon/4,$ centered at points of $X$ let
\[g\paren{\epsilon}=\max_{\textbf{y}\subseteq \textbf{x}^\epsilon}\abs{I_{i,\epsilon}\paren{\textbf{y}}}\,.\]
\begin{Proposition}
\label{prop_upper1}
Let $X$ be a bounded subset of a  metric space and let 
\[c_i\paren{X}=\limsup_{\epsilon\rightarrow 0}\frac{\log\paren{g\paren{\epsilon}}}{\log\paren{1/\epsilon}}\,.\]
 Then 
\[\text{dim}_{\PH}^i\paren{X}= c_i\paren{X}\,.\]
\end{Proposition}
\begin{proof}
Let $\alpha>c_i\paren{X},$ so there is a $D_1>0$ so that $g\paren{\epsilon}<\epsilon^{-\alpha}$ for all $\epsilon>0.$ By the definition of $g\paren{\epsilon},$ the hypotheses of Lemma~\ref{lemma_upper_new} are satisfied with $c=\alpha.$ Therefore, $\text{dim}_{\PH}^i\paren{X}\leq \alpha$ for all $\beta>c_i\paren{X}$ and $\text{dim}_{\PH}^i\paren{X}\leq c_i\paren{X}.$

Conversely, suppose that $\alpha<c_i\paren{X},$ and let $\alpha<\beta<c_i\paren{X}.$ For $D_2>0$ there exists a sequence $\epsilon_j \rightarrow 0$ so that $g\paren{\epsilon_j}>D_2\epsilon^{-\beta}$ for all $j\in\mathbb{N}.$ For each $j\in\mathbb{N}$ there exists a $\textbf{y}_j\subset \textbf{x}^{\epsilon_j}$ so that  
\[\abs{I_{i,\epsilon_j}\paren{\textbf{y}_j}}>D_2\epsilon_j^{-\beta}\,.\]
Then
\[E_\alpha^i\paren{\textbf{y}_j} \geq  \epsilon_j^\alpha \abs{I_{i,\epsilon_j}\paren{\textbf{y}_j}}\geq D_2 \epsilon_j^{\alpha-\beta}\,,\]
which limits to $\infty$ as $\epsilon_j\rightarrow 0$ because $\alpha<\beta.$ Therefore,  $\text{dim}_{\PH}^i\paren{X}\geq \alpha$ for all $\alpha>c_i\paren{X}$ and $\text{dim}_{\PH}^i\paren{X}= c_i\paren{X}.$
\end{proof}

In our next proposition, we show a relationship between the $\PH_i$ dimension and the upper box dimension.

\begin{Proposition}
\label{prop1}
Let $X$ be a bounded metric space, and let $\phi_X^i\paren{n}$ be the maximal number of $\PH_i$ intervals of a set of $n$ points in $X.$ If  $\phi_X^i\paren{n}= O\paren{n^\lambda}$ then
\[\text{dim}_{\PH}^i\paren{X}\leq \lambda\;\text{dim}_{\text{box}}\paren{X}\,.\]
\end{Proposition}

\begin{proof}
Let $\alpha>\text{dim}_{\text{box}}\paren{X},$ so there exists a $C'>0$ so that
\begin{equation}
\label{eqn_prop1a}
\abs{\textbf{x}^\epsilon}<C'\epsilon^{-\alpha}
\end{equation}
for all $\epsilon>0.$ Also, there is a $D_3>0$ so that  
\begin{equation}
\label{eqn_prop1b}
\phi_{X}^i\paren{n}<D_3 n^\lambda
\end{equation}
 for all $n>0.$ Then, if $\textbf{y}\subseteq\textbf{x}^\epsilon,$ 
\begin{align*}
\abs{I_{i,\epsilon}\paren{\textbf{y}}}\leq & \;\; \abs{PH_i\paren{\textbf{y}}}\\
< & \;\; D_3\abs{\textbf{y}}^\lambda  &&\text{by Eqn.~\ref{eqn_prop1b}}\\
\leq & \;\; D_3\abs{\textbf{x}^\epsilon}^\lambda &&\text{because $\textbf{y}\subseteq \textbf{x}^\epsilon$}\\
< & \;\; D_3\paren{C'  \epsilon^{-\alpha}}^\lambda  &&\text{by Eqn.~\ref{eqn_prop1a}}\\
\leq & \;\;  D_3 C' \epsilon^{-\lambda\alpha}\,,
\end{align*}
and $g\paren{\epsilon}<D_3 C'  \epsilon^{-\lambda\alpha}$ for all $\epsilon>0.$ By the previous lemma,
\begin{align*}
\text{dim}_{\PH}^i\paren{X}=&\;\;\limsup_{\epsilon\rightarrow 0}\frac{\log\paren{g\paren{\epsilon}}}{\log\paren{1/\epsilon}}\\
\leq & \;\; \lim_{\epsilon\rightarrow 0} \frac{\log\paren{D_3 C'  \epsilon^{-\lambda\alpha}}}{\log\paren{1/\epsilon}}\\
= & \;\; \lim_{\epsilon\rightarrow 0} \frac{-\lambda\alpha \log\paren{\epsilon}}{\log\paren{1/\epsilon}}\\
=&\;\; \lambda\alpha
\end{align*}
for all $\alpha>\text{dim}_{\text{box}}\paren{X}.$ Therefore, 
\[\text{dim}_{\PH}^i\paren{X}\leq \lambda\;\text{dim}_{\text{box}}\paren{X}\,.\]
\end{proof}

If $X\subset \mathbb{R}^m$ we can apply the well-known Upper Bound Theorem on the maximal number of simplices of a Delaunay triangulation to bound $\phi_X^i\paren{n}.$
\begin{theorem}[The Upper Bound Theorem~\cite{stanley1975,1970mcmullen}]
The maximum number of simplices of a Delaunay triangulation with $n$ vertices in $\R^m$ is
\[\binom{n-\floor{\frac{m+1}{2}}}{n-m}+\binom{n-\floor{\frac{m+2}{2}}}{n-m}=O\paren{n^{\floor{\frac{m+1}{2}}}}\,.\]
Furthermore, if $0<i<\floor{\frac{m+1}{2}}$ the maximum number of $i$-simplices of a Delaunay triangulation with $n$ vertices in $\R^m$ is
\[f_i\paren{m,n}=\binom{n}{i+1}=O\paren{n^{i+1}}\,.\]
\end{theorem}
The upper bound is sharp, and is acheived by taking the Delaunay triangulation on the vertices of a cyclic polytope~\cite{2006amenta}. We have the following corollary.
\begin{cor}
\label{corollary_upper}
Let $X$ be a bounded subset of $\R^m$ then
\[\text{dim}_{\PH}^i\paren{X}\leq \min\paren{\left(i+1\right),\floor{\frac{m+1}{2}}}\text{dim}_{\text{box}}\paren{X}\,.\]
 In particular, if $m=2$ 
\[\text{dim}_{\PH}^1\paren{X}\leq \text{dim}_{\text{box}}\paren{X}\,.\]
\end{cor}
\begin{proof}
Let $\textbf{x}$ be a finite subset of $\R^m.$ If necessary, we may perturb the points of $\textbf{x}$ by an amount less than $\delta/2,$ where $\delta$ is the minimum length of a $\PH_i$ interval of $\textbf{x}$ to place the points in general position without decreasing the number of intervals (by Corollary~\ref{stableCorollary}). The number of $\PH_i$ intervals of $\textbf{x}$ is bounded above by the number of $i$-simplices in the Alpha complex on $X,$ which equals the number of $i$-simplices in the Delaunay triangulation. Therefore, by the Upper Bound Theorem,
\[\abs{\PH_i\paren{\textbf{x}}}\leq \binom{n}{i+1}=O\paren{n^{\lambda}}\,,\]
where $\lambda=\min\paren{\left(i+1\right),\floor{\frac{m+1}{2}}}.$ 

It follows that if $X$ is a bounded subset of $\mathbb{R}^m$ then
\[\phi_X^i\paren{n}=O\paren{n^{\lambda}}\,,\]
and the desired result follows from Proposition~\ref{prop1}.
\end{proof}
This completes the proof of the upper bound in Theorem~\ref{mainTheorem}.

\subsection{An Example with Many Intervals}
\label{sec:Arcs}

\begin{figure}
\center
\includegraphics[width=4cm]{%
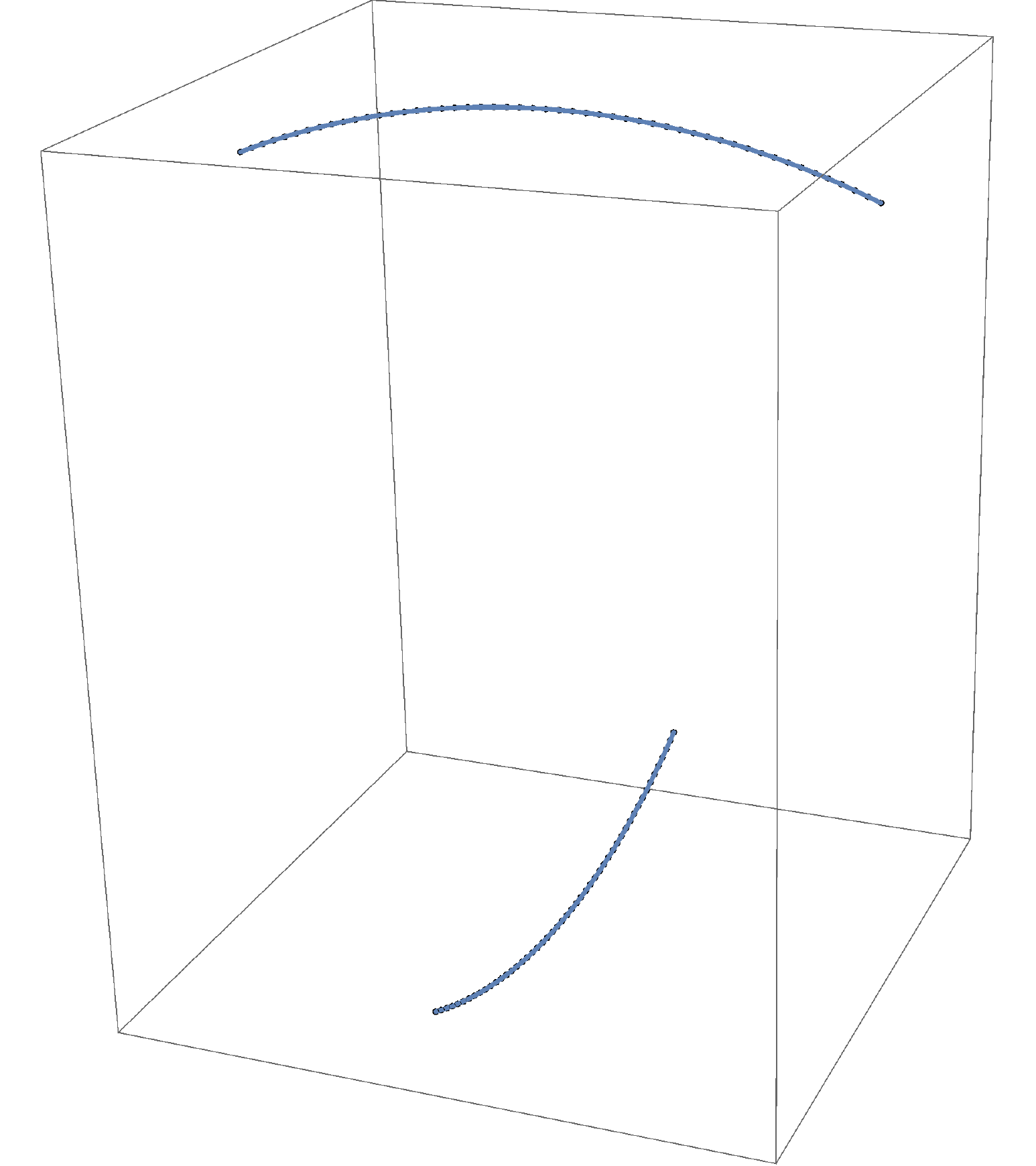}
\caption{\label{fig:arcs}The set $X:$ two opposing arcs from unit circles.}
\end{figure}
We present  an example of a subset of $\R^3$ with subsets of size $n$ that appear to have $\approx n^{1.5}$ $\PH$ intervals. This example was suggested by Herbert Edelsbrunner.

Let $C_1$ and $C_2$ be the following arcs from unit circles in $\R^3:$
\[C_1=\set{\paren{\cos\paren{\theta},\sin\paren{\theta},0}:\theta\in\brac{-\pi/8,\pi/8}}\]
\[C_2=\set{\paren{1-\cos\paren{\theta},0,\sin\paren{\theta}}:\theta\in\brac{-\pi/8,\pi/8}}\]
 and let $X=C_1\cup C_2.$ $X$ is shown in Figure~\ref{fig:arcs}. Let $\textbf{x}_n\subset X$ be a point set obtained by placing $\floor{n/2}$ uniformly spaced points on each of the two circular arcs. Computations indicate that $\abs{\PH_i\paren{\textbf{x}_n}}\approx\abs{\textbf{x}_n}^{1.5}$ for $i=1,2$ (Figure~\ref{fig:IntervalNumber}), but that $E_1^1\paren{\textbf{x}_n}$ and $E_1^2\paren{\textbf{x}_n}$ are bounded as functions of $n$ (Figure~\ref{fig:IntervalSum}). As such, the question remains of whether the $\PH_i$ dimension is bounded above by the upper box dimension for any subset of Euclidean space. This example suggests that it may be necessary to bound the number of ``long'' intervals of a ``well-spaced'' point set rather than the total number of intervals to achieve a sharp result.

\begin{figure}
\center
\subfigure[]{%
	\label{fig:IntervalNumber}{%
		\includegraphics[width=6cm]{%
			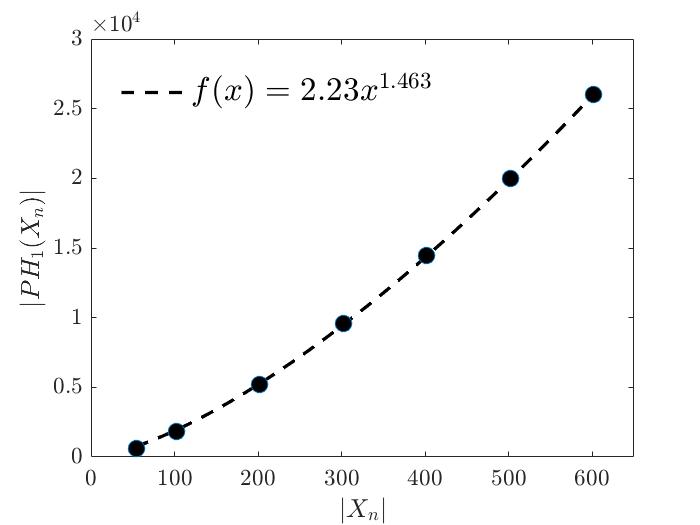}}}
\hspace{20pt}
\subfigure[]{%
	\label{fig:IntervalSum}{%
		\includegraphics[width=6cm]{%
			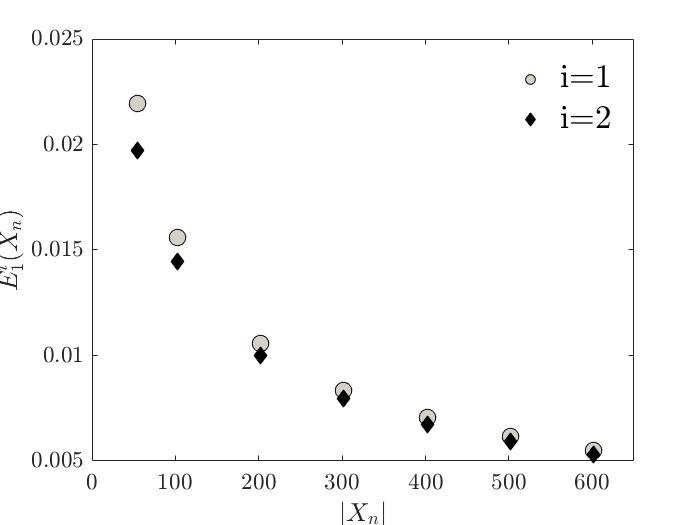}}}
\caption{\label{fig:arcPH}(a) The number of $\PH_1$ intervals of $\textbf{x}_n,$ which appears to grow approximately as $n^{1.5}.$ The data for $\PH_2$ is similar. (b) The quantities $E_1^1\paren{\textbf{x}_n}$ and $E_1^2\paren{X_n},$ which appear to be bounded as a function of $n.$ Persistent homology was computed by calculating the alpha complex with CGAL~\cite{CGAL} and passing the resulting filtration to JPLEX~\cite{jplex}.}
\end{figure}

\section{A Lower Bound}
\label{sec_lower_bound}

In this section, we prove the lower bounds in Theorems~\ref{mainTheorem} and~\ref{thmLower}. However, we first consider the case of a bounded subset of Euclidean space with non-empty interior. Our later arguments will be more complicated, but have the same general outline.

\begin{Proposition}
\label{prop:interior}
If $X$ is a bounded subset of $\mathbb{R}^m$ with non-empty interior then 
\s{\text{dim}_{\PH}^i\paren{X}= m}
for $i=0,\ldots,m-1.$
\end{Proposition}
\begin{proof}
The upper bound was shown in Corollary~\ref{corollary_upper_2}.

Let $0\leq i \leq m-1.$ First, we compute the $i$-dimensional persistent homology of the vertex set of the regular $i+1$-simplex $\sigma$ with edge length $\sqrt{2}.$ We may assume that the simplex is formed by the standard basis vectors $e_1,\ldots,e_{i+1}$ in $\mathbb{R}^{i+1}.$ The circumcenter of $\sigma$ is $c\coloneqq \paren{\frac{1}{i+1},\ldots,\frac{1}{i+1}},$ and the circumradius is
\[d\paren{e_j,c}=\sqrt{\paren{1-1/i}^2+\frac{i}{\paren{i+1}^2}}=\sqrt{\frac{i}{i+1}}\,,\]
so $\sigma$ enters the alpha complex $A_{\epsilon}\paren{\sigma}$ when $\epsilon=\sqrt{\frac{i}{i+1}}.$ The same computation shows that the $i$-dimensional faces of $\sigma$ enter  $A_{\epsilon}\paren{\sigma}$ when $\epsilon=\sqrt{\frac{i-1}{i}}=\sqrt{1-\frac{1}{i}}.$ Therefore, $\PH_i\paren{\sigma}$ consists of the single interval $\paren{\sqrt{1-\frac{1}{i}} , \sqrt{\frac{i}{i+1}}}.$

For convenience, let $\tau$ be the length of the single $\PH_1$ interval of the regular $i+1$-simplex with edge length $1,$
\[\tau=\frac{1}{\sqrt{2}}\paren{\sqrt{\frac{i}{i+1}}-\sqrt{1-\frac{1}{i}}}\]
and $\hat{d}$ be the death-time of that interval,
\[\hat{d}=\sqrt{\frac{i}{2i+2}}\,.\]

 $X$ has non-empty interior so we can find an $m$-dimensional cube $C\subset X.$ Rescale $X$ if necessary so that $C$ is a unit cube. Fix $n\in\mathbb{N}$ and sub-divide  $C$ into $n^m$ cubes of width $1/n.$ Let $C_1,\ldots, C_t$ be a maximal sub-collection of these cubes so that 
\begin{equation}
\label{eq:lemmaEuc}
d\paren{C_k,C_l}> 2\frac{\hat{d}}{n}
\end{equation}
 for all $k,l\in\set{1,\ldots,t}$ so that $k\neq l$ (where $d\paren{C_k,C_l}$ is the minimum distance between any pair of points in the two cubes). There is a constant $F$ depending only on $i$ and $m$ so that $t>F n^m$ for all sufficiently large $n.$ 

For $j=1,\ldots, t$ find a regular $i+1$-simplex of edge length $\frac{1}{n}$ inside each cube $C_j.$ Let $\textbf{x}_j$ be the vertices of this simplex, and let $\textbf{z}_n=\cup_{j=1}^{t} \textbf{x}_j.$ By Equation~\ref{eq:lemmaEuc}, the \v{C}ech complex on $\textbf{z}_n$ equals the disjoint union of the \v{C}ech complexes on $\textbf{x}_1,\ldots, \textbf{x}_t$ for any filtration value less than $\frac{\hat{d}}{n}$ (the first time an edge between a point in $\textbf{x}_j$ and one in $\textbf{x}_k$ can enter the complex for $j\neq k).$  $\PH_i\paren{\textbf{x}_j}$ consist of a single interval of length $\frac{1}{n}\tau$ that dies at time $\frac{\hat{d}}{n}.$ It follows that $\PH_i\paren{\textbf{z}_n}$ contains at least $F n^m$ intervals of length $\frac{1}{n}\tau.$ 

Let $\alpha<m,$ and compute
\[E_\alpha^i\paren{\textbf{z}_n}\geq F n^m\paren{\frac{\tau}{n}}^\alpha=F\tau^{\alpha}n^{m-\alpha}\,\]
which limits to $\infty$ as $n\rightarrow\infty,$ because $\alpha<m.$ Therefore, $\text{dim}_{\PH}^i\paren{X}\geq \alpha$ for any $\alpha<m$ and $\text{dim}_{\PH}^i\paren{X}=m,$ as desired.
\end{proof}

In the preceding proof, we could find a cube in Euclidean space for which each sub-cube was occupied, and found point sets with non-trivial homology inside those cubes. If $X$ is a bounded subset of $\mathbb{R}^m,$ and $\text{dim}_{\text{box}}{X}<m,$ most cubes will not be occupied. As such, we need to proceed with care. 

Let $\brac{N}$ denote the integers $1,\ldots,n$ and let $\brac{N}^m\subset \Z^m$ be $\brac{N}\times\brac{N}\times\ldots\times\brac{N}.$ For each $x\in\delta \mathbb{Z}^m,$ let the cube corresponding to $x$ be the cube of width $\delta$ centered at $x.$ The grid of mesh $\delta$ is the set of all cubes in this tessellation.  If $X$ is a bounded subset of $\R^m,$ the upper box dimension of $X$ controls the number of cubes in the grid of mesh $\delta$ that intersect $X.$  Our strategy to prove a lower bound for the $\PH$ dimension of $X$ in terms of the upper box dimension is to show that a sufficiently large collection of points, each in a distinct cube of $\brac{N}^m,$ must contain a subset with non-trivial persistent homology. 

\begin{figure}
\centering  
\subfigure[]{\includegraphics[width=0.2\linewidth]{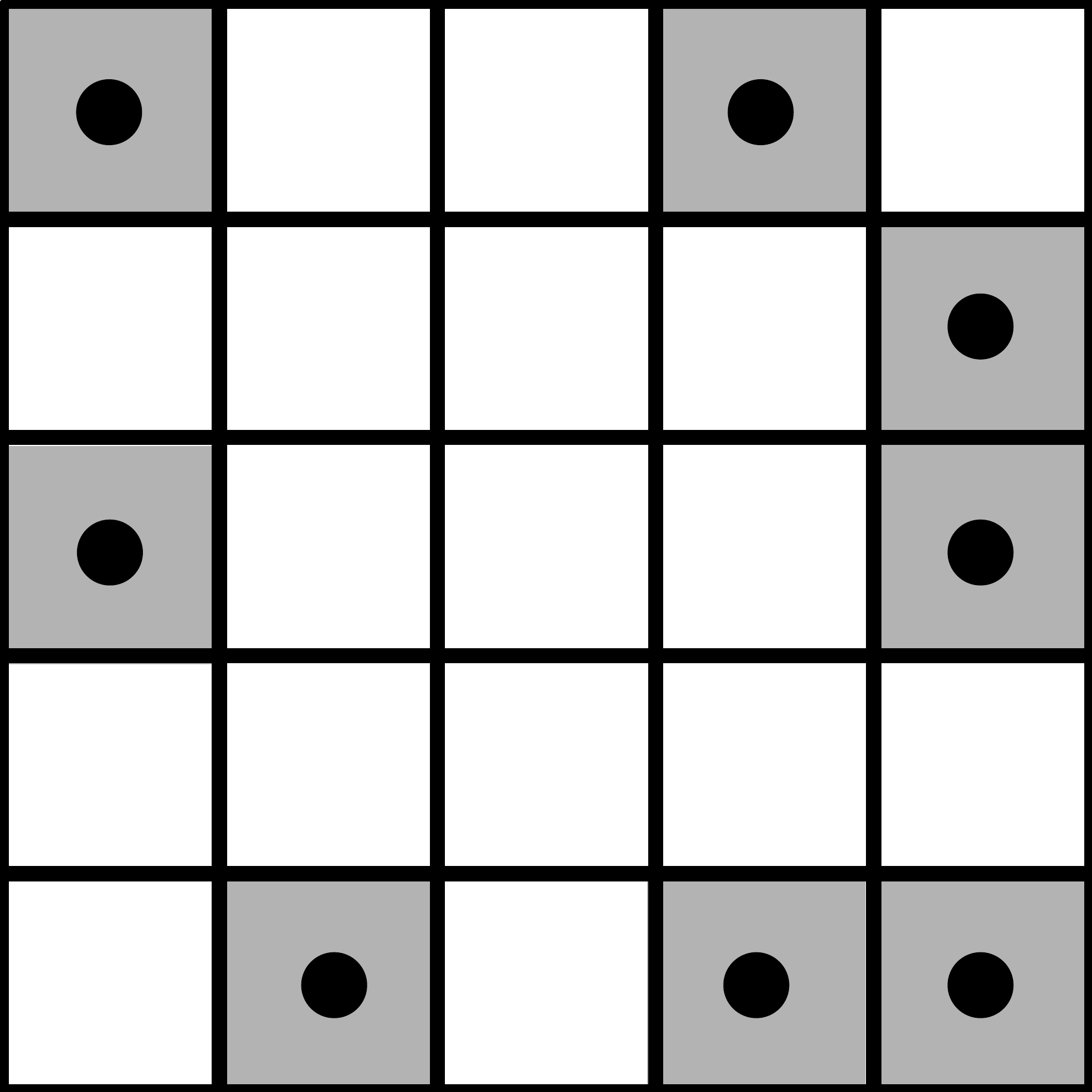}}
\subfigure[]{\includegraphics[width=0.2\linewidth]{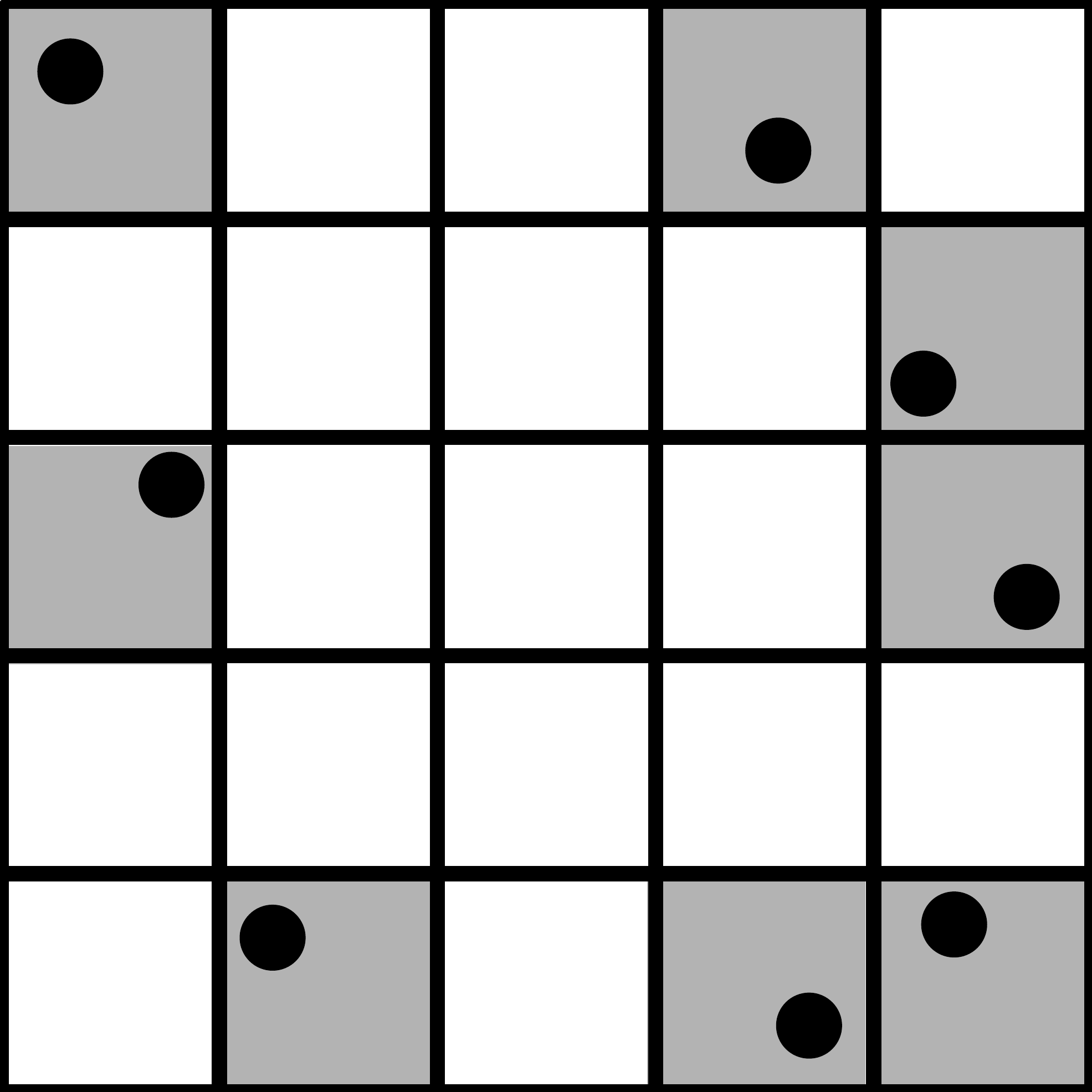}}
\subfigure[]{\includegraphics[width=0.2\linewidth]{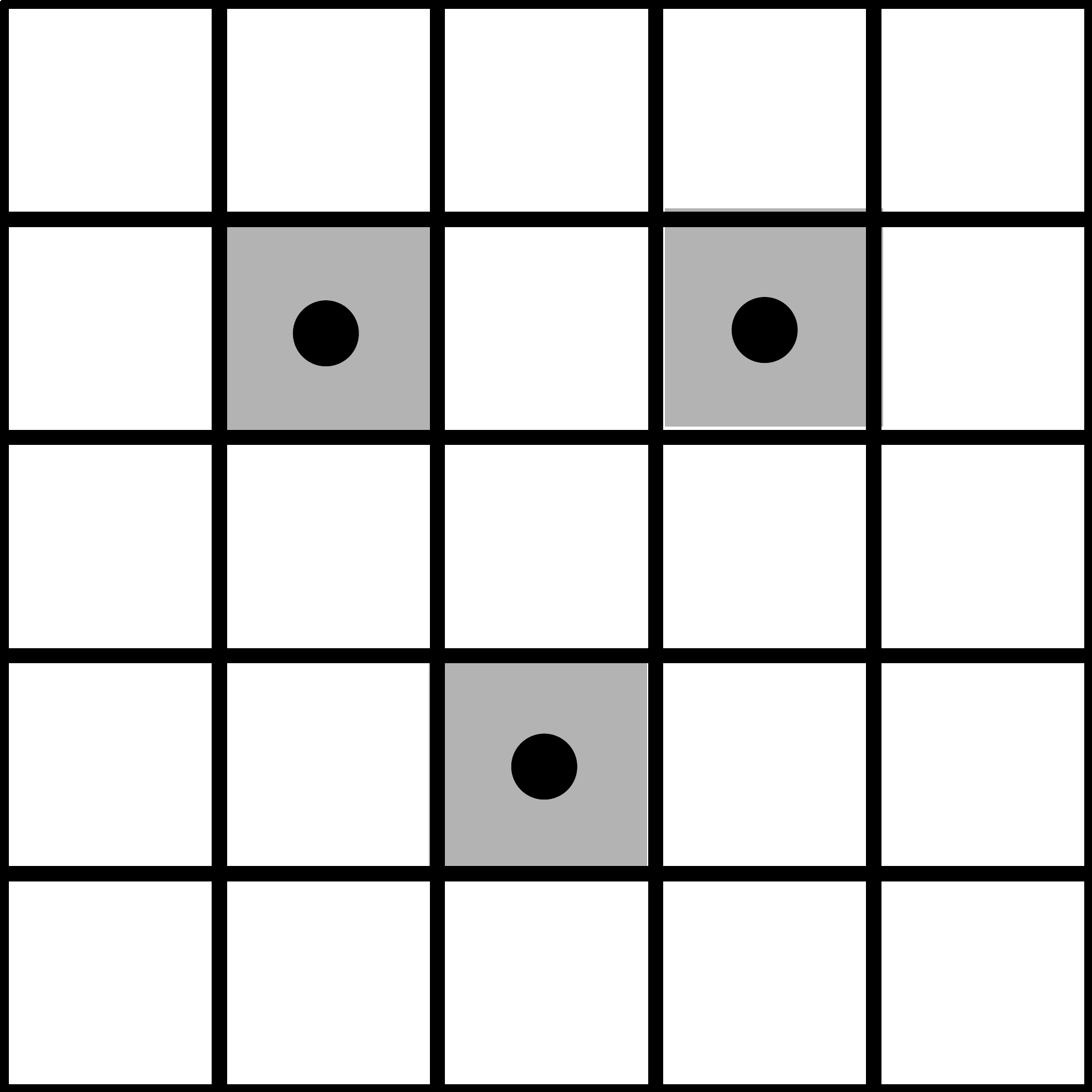}}
\subfigure[]{\includegraphics[width=0.2\linewidth]{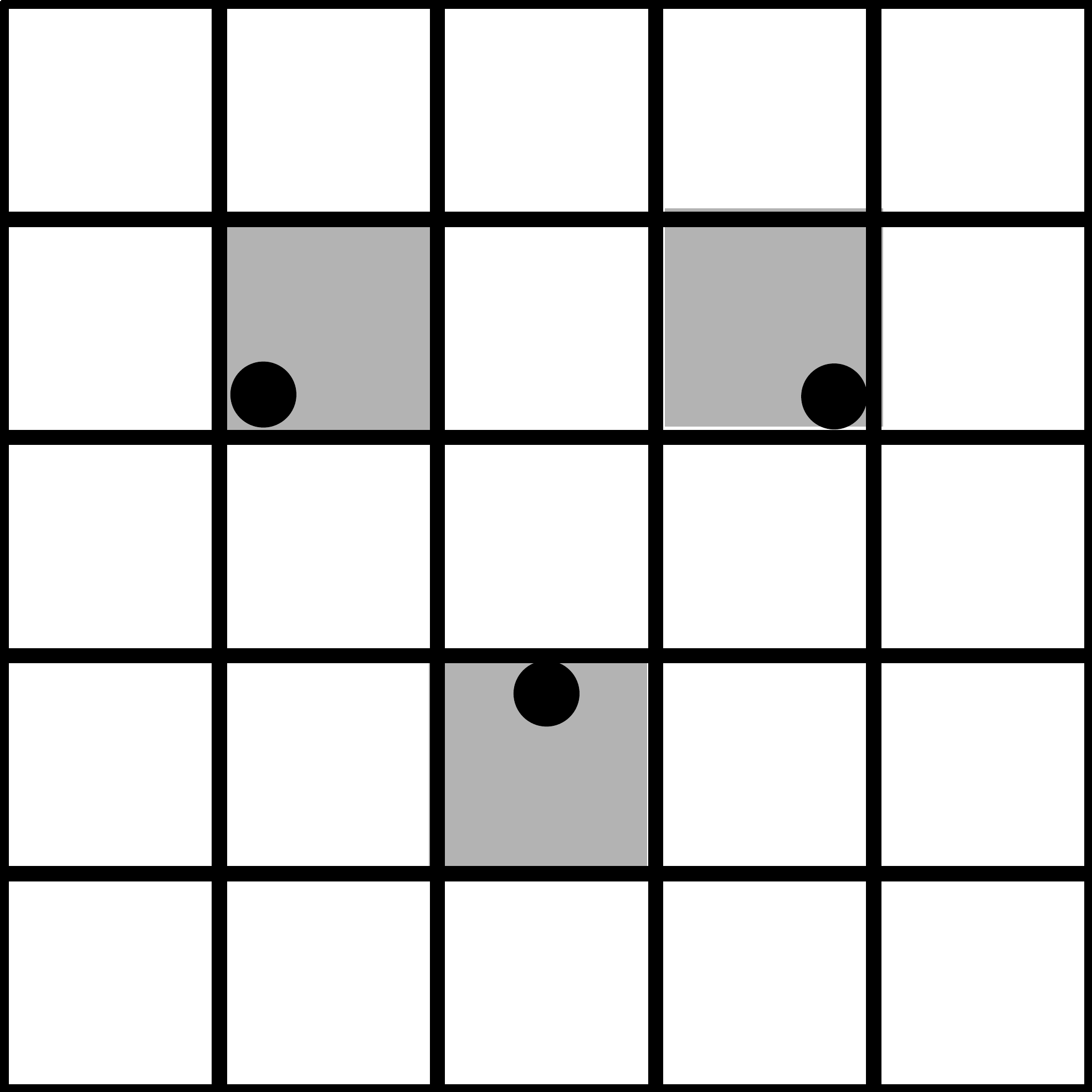}}
	\caption{\label{fig:stable}The $\PH_1$ class of the lattice points corresponding to the gray cubes in (a) and (b) is stable --- 
	any choice of one point in each cube will give the vertices of an acute triangle, and therefore a set with non-trivial $\PH_1.$ The one in (c) and (d) is not, because the points in (d) form an obtuse triangle so the persistent homology of that set is trivial.}
\end{figure}

\begin{Definition}
\label{defn_stable}
$\textbf{x}\subset \Z^m$ has a \textbf{stable} $\PH_i$-class if any point collection $\textbf{y}$ consisting of exactly one point in each cube corresponding to a point of $\textbf{x}$ has a $\PH_i$ interval $I$ so that $\abs{I}\geq c$ for some constant $c>0$ (see Figure~\ref{fig:stable}). The supremal such $c$ is called the \textbf{size} of the stable persistence class.
\end{Definition}

The stability of the bottleneck distance immediately implies the following.
\begin{Proposition}
\label{prop:stable}
If $\textbf{x}\subset \Z^m$ and there is an interval $I\in \PH_i\left(\textbf{x}\right)$ with
\s{\abs{I}>\sqrt{m}+\epsilon\,,}
then $\textbf{x}$ has a stable $\PH_i$-class of size at least $\epsilon.$
\end{Proposition}
\begin{proof}
If $\textbf{x}$ and $\textbf{y}$ are in Definition~\ref{defn_stable}, then $d_H\paren{\textbf{x},\textbf{y}}<\sqrt{m}/2,$ the maximal distance between a point in a unit cube and the center of that cube. By bottleneck stability, $\PH_i\paren{\textbf{y}}$ has an interval of length at least $\epsilon.$
\end{proof}

\begin{Definition}
\label{defn_nontriviality}
Let $\xi_i^m\left(N\right)$ be the size of the largest subset $\textbf{x}$ of $\brac{N}^m$ so that no subset $\textbf{y}$ of $\textbf{x}$ has a stable $\PH_i$-class of size greater than 1. That is,  $\xi_i^m\left(N\right)$ is the smallest integer so that if $\abs{\textbf{x}}> \xi_i^m\left(N\right),$ then there exist a $\textbf{y}\subset \textbf{x}$ with a stable $\PH_i$ class. Define
\[\gamma_i^m=\liminf_{N\rightarrow\infty}\frac{\log\paren{\xi_i^m\paren{N}}}{\log\paren{N}}\,.\]
\end{Definition} 

The following proposition is stated for the \v{C}ech complex and $\gamma_i^m$ is defined for that filtration, but an analogous result holds for the Rips complex (though the corresponding constant $\hat{\gamma}_i^m$ may be different).
\begin{Proposition}
\label{prop:kappaprop}
If $X$ is a bounded subset of $\R^m$ and $\text{dim}_{\text{box}}\paren{X}>\gamma_i^m$
then 
\s{\text{dim}_{\PH}^1\paren{X}\geq \text{dim}_{\text{box}}\paren{X}\,.}
\end{Proposition}

Before proving this proposition, we introduce notation and state a technical lemma related to box-counting. The notation is illustrated in Figure~\ref{fig:boxNotation}. For $\delta\in \mathbb{R},$ let $\mathcal{B}_{\delta}$ be the collection of cubes in the grid of mesh $\delta$ in $\mathbb{R}^m$ that intersect $X.$ Also, if $k\in \mathbb{N},$ consider the larger cubes  $\mathcal{B}_{k\delta}.$ Note that each cube of $\mathcal{B}_{\delta}$ is contained in a unique cube of $\mathcal{B}_{k\delta},$ and each cube of $\mathcal{B}_{k\delta}$ contains at least one cube of $\mathcal{B}_{\delta}.$ If $C\in\mathcal{B}_{k\delta},$ let
\[\mathcal{C}_k\paren{C}=\set{B\in \mathcal{B}_{\delta}:B\subset C}\,.\]

\begin{figure}
\centering
\includegraphics[width=.7\textwidth]{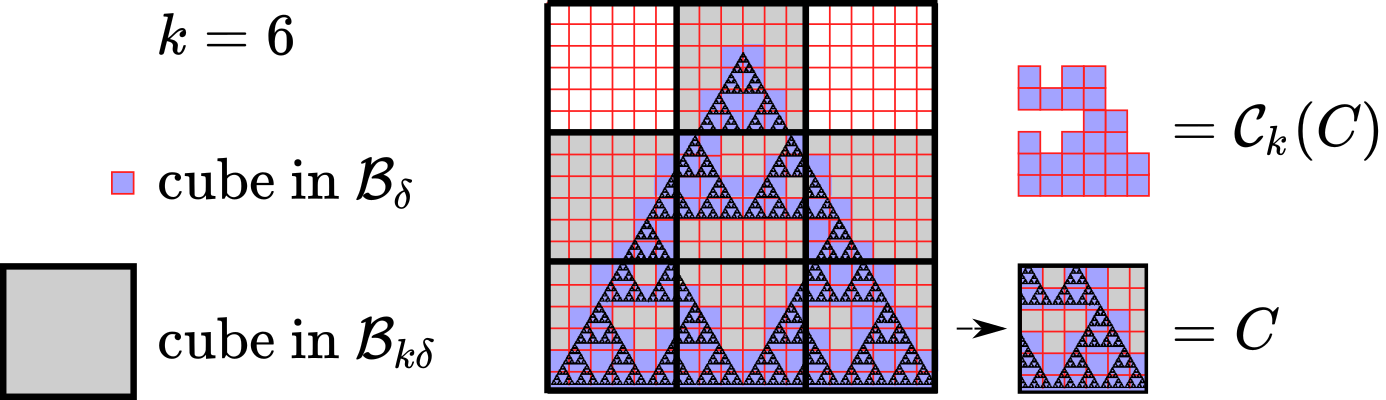}
\caption{The box notation used in the proofs of Proposition~\ref{prop:kappaprop} and Lemma~\ref{lemma:insidebox}.}
\label{fig:boxNotation}
\end{figure}

Furthermore, for $\lambda>0$ define
\[\mathcal{D}^\lambda_{\delta,k}=\set{C\in\mathcal{B}_{k\delta}:\abs{\mathcal{C}_k\paren{C}}>k^{\lambda}}\,.\]

The idea of the proof of Proposition~\ref{prop:kappaprop} is to choose $k$ and $\lambda$ so that we can find a set of points in every cube $\mathcal{D}^\lambda_{\delta,k}$ that has a stable $\PH_i$ class, and show that $\delta^\alpha \abs{\mathcal{D}^\lambda_{\delta,k}}\rightarrow\infty$ for $\alpha<\text{dim}_{\text{box}}\paren{X}.$ This is a little tricky as, the definition of the upper box dimension only gives us weak control over the asymptotics of $\abs{\mathcal{B}_{\delta}}.$ As such, we cannot choose a fixed value of $k$ and must instead let $k$ increase slowly as $\delta\rightarrow 0.$

 \begin{Lemma}
\label{lemma:insidebox}
If $X$ is a bounded subset of $\R^m$ and $\alpha<\lambda<\text{dim}_{\text{box}}\paren{X},$ there exist a sequence of real numbers $\delta_j\rightarrow 0$ and a sequence of natural numbers $k_j\rightarrow \infty$ so that
\[\delta_j^{\alpha} \abs{\mathcal{B}^{\lambda}_{\delta_j,k_j}}\rightarrow\infty\,.\]
\end{Lemma}

We prove Proposition~\ref{prop:kappaprop} first, then return to the proof of Lemma~\ref{lemma:insidebox}.

\begin{proof}[Proof of Propostion~\ref{prop:kappaprop}]
Let $d=\text{dim}_{\text{box}}\paren{X}$ and $\alpha<d.$ We will find finite point collections $\textbf{z}_j\subset X$ so that $E^i_{\alpha}\paren{\textbf{z}_j}\rightarrow\infty$ as $j\rightarrow\infty.$  Choose $\lambda$ so that 
\[\text{max}\paren{\gamma_i^m,\alpha}<\lambda<d\,.\]
By the previous lemma, there are sequences of real numbers $\delta_j\rightarrow 0$ and $k_j\rightarrow \infty$ so that
\[\delta_j^{\alpha} \abs{\mathcal{D}^{\lambda}_{\delta_j,k_j}}\rightarrow\infty\]
Each cube $C\in \mathcal{D}^{\lambda}_{\delta_j,k_j}$ contains at least $k_j^{\lambda}$ sub-cubes of width $\delta_j$ that intersect $X.$ $\lambda>\gamma_i^m,$ so for sufficiently large $k_j$ we can find a finite point set $\textbf{x}\paren{C}\in X\cap C$ so that $\PH_i\paren{\textbf{x}\paren{C}}$ has an interval of length at least $\delta_j.$ To combine the contributions of multiple points sets $\textbf{x}\paren{C}$, we must thin the collection $\mathcal{D}^{\lambda}_{\delta_j,k_j}.$ 

Let $\set{C_1,\ldots,C_t}$ be a maximal collection of cubes in  $\mathcal{D}^{\lambda}_{\delta_j,k_j}$ so that $d\paren{C_k,C_l}> 2 k_j\delta_j \sqrt{m}$ for all $k,l\in\set{1,\ldots,t}$ so that $k\neq l$ (where $d\paren{C_k,C_l}$ is the minimum distance between any pair of points in the two cubes). There is a constant $F$ depending only on the ambient dimension $m$ so that $t\geq F \abs{\mathcal{D}^{\lambda}_{\delta_j,k_j}}$ for sufficiently large $j.$  

Let $\textbf{z}_j=\cup_{k=1}^t\textbf{x}\paren{C_k}.$ By construction, the \v{C}ech complex on $\textbf{z}_j$ equals the disjoint union of the \v{C}ech complexes  on $\textbf{x}\paren{C_1},\ldots, \textbf{x}\paren{C_t}$ for any filtration value less than $k_j\delta_j\sqrt{m}.$ Furthermore, the diameter of each set $\textbf{x}\paren{C_k}$ is less than $k_j\delta_j\sqrt{m}$ so any interval $(b_0,d_0)\in \PH_i\paren{\textbf{x}\paren{C_k}}$ satisfies $d_0<k_j\delta_j\sqrt{m}.$ Therefore, as sets of intervals, we have that $\cup_k PH_i\paren{ \textbf{x}\paren{C_k}}\subseteq \PH_i\paren{\textbf{z}_j}.$ It follows that $\PH_i\paren{\textbf{z}_j}$ contains at $t$ intervals of length greater than $\delta_j.$

Therefore,
\[E_{\alpha}^i\paren{\textbf{z}_j}\geq t \delta_j^{\alpha} \geq  F \abs{\mathcal{D}^{\lambda}_{\delta_j,k_j}} \delta_j^{\alpha}\,,\]
which goes to $\infty$ as $j\rightarrow \infty$ by the previous lemma. Thus $\text{dim}_{\PH}^i\paren{X}\geq\alpha$ for all $\alpha < d,$ and $\text{dim}_{\PH}^i\paren{X}\geq d$  as desired.
\end{proof}

\begin{proof}[Proof of Lemma~\ref{lemma:insidebox}]
Let $d=\text{dim}_{\text{box}}\paren{X},$ and choose $0<p<1$ so that $d-pm>\alpha.$ Also, choose
\[\epsilon<\text{min}\paren{2\paren{d-pm-\alpha},pd-p\lambda}\,.\]
 Let $\beta=d+\epsilon/2,$ and note for future reference that 
\begin{equation}
\label{eqn_a1}
\alpha+pm-\beta+\epsilon=\alpha+pm-d+\epsilon/2 <0
\end{equation}
and 
\begin{equation}
\label{eqn_a2}
p\beta-p\lambda-\epsilon=pd-p\lambda -\paren{1-p/2}\epsilon > pd-p\lambda-\epsilon>0\,.
\end{equation}

$\beta>d$ so there exists a $\delta_0>0$ so that $\abs{B_\delta}<\delta^{-\beta}$ for all $\delta<\delta_0.$ Also, $\beta-\epsilon<d$ so there exists a sequence $\delta_j\rightarrow 0$ so that $\abs{B_{\delta_j}}>\delta_j^{-\beta+\epsilon}$ for all $j\in\N.$ We may assume that $\delta_j<\delta_0$ for all $j.$ Let 
\begin{equation}
\label{eqn_a3}
k_j=\floor{\delta_j^{-p}}\,.
\end{equation}

A cube in $\mathcal{B}_{k_j\delta_j}$ can contain at most $k_j^m$ cubes in $\mathcal{B}_{\delta_j},$ and every cube in $\mathcal{B}_{\delta_j}$ is contained in exactly one cube $\mathcal{B}_{k_j\delta_j}.$ It follows  that $\mathcal{D}^{\lambda}_{\delta_j,k_j}$ is greater than or equal to the smallest integer $a$ so that
\[a k_j^m + \paren{\abs{\mathcal{B}_{k_j\delta_j}}-a}k_j^{\lambda}\geq \abs{B_{\delta_j}}\,.\]
Rearranging terms, we have that
\begin{align*}
a\geq & \;\; \frac{\abs{B_{\delta_j}}-\abs{B_{k_j\delta_j}}k_j^{\lambda}}{k_j^m-k_j^{\lambda}} \\
 =  &\;\;  \frac{\abs{B_{\delta_j}}k_j^{-\lambda}-\abs{B_{k_j\delta_j}}}{k_j^{m-\lambda}-1} \\
\geq  & \;\; \frac{\delta_j^{-\beta+\epsilon}k_j^{-\lambda}-k_j^{-\beta}\delta_j^{-\beta}}{k_j^{m-\lambda}-1} \\
 =  & \;\;  \frac{\delta_j^{-\beta+\epsilon}\floor{\delta_j^{-p}}^{-\lambda}-\floor{\delta_j^{-p}}^{-\beta}\delta_j^{-\beta}}{\floor{\delta_j^{-p}}^{m-\lambda}-1} &&\text{by Eqn.~\ref{eqn_a3}} \\
 \approx  & \;\; \frac{\delta_j^{p\lambda-\beta+\epsilon}-\delta_j^{p\beta-\beta}}{\delta^{p\lambda-pm}-1} &&\text{as $\delta_j\rightarrow 0$}\\
 \approx & \;\; \delta_j^{pm-\beta+\epsilon}\,,\
\end{align*}
where we used that $p\lambda-\beta+\epsilon>p\beta-\beta$ by Equation~\ref{eqn_a2}.

 Choose a $J$ so that if $j>J,$
\[\abs{\mathcal{D}^{\lambda}_{\delta_j,k_j}} \geq  \frac{1}{2}  \delta_j^{pm-\beta+\epsilon}\,.\] Then, if $j>J,$ 

\[\delta_j^{\alpha} \abs{\mathcal{D}^{\lambda}_{\delta_j,k_j}}  \geq  \frac{1}{2}  \delta_j^{\alpha} \delta_j^{pm-\beta+\epsilon} = \frac{1}{2}\delta_j^{\alpha+pm-\beta+\epsilon}\,,\]
which by equation Equation~\ref{eqn_a1}, limits to $\infty$ as $j\rightarrow \infty.$
\end{proof}

\subsection{A Bound for $\gamma_1^m$}
We prove an upper bound for $\gamma_1^m$ by considering the contribution of small triangles to the persistent homology of a subset of $\Z^m.$ Three points in Euclidean space give rise to a $\PH_1$-class if and only if they are the vertices of an acute triangle, in which case the total persistence (death minus birth) equals the circumradius minus half the length of its longest edge.

\begin{Proposition}
\label{triangleLemma}
\[\gamma_1^m\geq m-1/2\,.\]
\end{Proposition}
\begin{proof}
We will prove that $\xi_1^m\paren{N} \in O\paren{N^{m-1/2}}$ by showing that if $X\subset \brac{N}^m,$ 
\[X\subset\brac{N}^m,\qquad c>\sqrt{4\sqrt{m}+4}, \qquad \abs{X}> 8c N^{m-1/2}+4N^{m-1},\]
 and $N$ is sufficiently large then there exist three points $p_1, q_2,$ and $q_3$ of $X$ so that $\PH_1\paren{\set{p_1,q_2,q_3}}$ contains an interval of length at least $1+\sqrt{m}.$ The desired result will then follow from Proposition~\ref{prop:stable}.

 First, we reduce the problem to a two-dimensional one. If $m>2,$ we may find a two-dimensional slice $S$ of $\brac{N}^m$ of the form $\brac{N}^2 \times w,$ so that the density of $X$ in $S$ is greater than or equal to the density of $X$ in $\brac{N}^m.$ That is,
\begin{equation}
\label{eqn_tri1}
\abs{X\cap S}\geq N^2\frac{\abs{X}}{N^m} > 8cN^{1.5}+4N\,.
\end{equation}
We will treat $S$ as a two-dimensional grid of $m$-dimensional cubes, with $N$ rows and $N$ columns, as in Figure~\ref{fig_half}.

\begin{figure}
\center
\includegraphics[width=150pt]{%
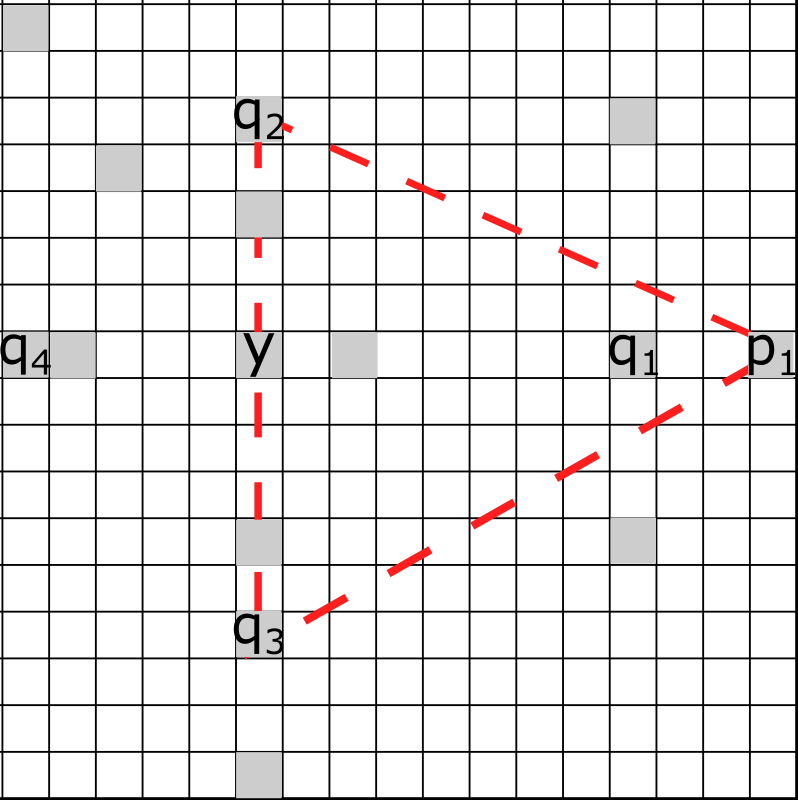}
\caption{\label{fig_half} The setup in Proposition~\ref{triangleLemma}. The points of $X$ are shown in gray, and the triangle formed by the points $p_1,$ $q_2,$ and $q_3$ is shown by the red dashed lines. It is an acute triangle, and the triangle $q_1,q_2,$ and $q_3$ is either acute or right.}
\end{figure}

Let $R_1,\ldots,R_N$ be the rows of $\brac{N}^2$ and $C_1\ldots C_N$ the the columns. Let $1\leq j \leq N.$ If $\abs{X\cap R_j}>4c\sqrt{N}+2,$ let $A_j$ be the set of $4c\sqrt{N}+2$ elements of $X\cap R_j$ that contains the leftmost $2c\sqrt{N}+1$ elements of $X\cap R_j$ and the  rightmost $2c\sqrt{N}+1$ elements  of $X\cap R_j.$ Otherwise, let $A_j=X\cap R_j.$ Similarly, if $\abs{X\cap C_j}>4c\sqrt{N}+2,$ let $B_j$ be the set of   $4c\sqrt{N}+2$ elements of $X\cap C_j$  that contains the uppermost $2c\sqrt{N}+1$ elements of $X\cap C_j$ and the  lowermost $2c\sqrt{N}+1$ elements of $X\cap C_j.$ Otherwise, let $B_j=X\cap C_j.$ Then, if $D=\bigcup_{j=1}^N A_j \cup B_j,$
\[\abs{D}\leq 2N \paren{4c\sqrt{N}+2}=8cN^{1.5}+4N < \abs{X\cap S}\,,\] 
by Equation~\ref{eqn_tri1}. Therefore, $X\cap S \setminus D\neq \varnothing$ and there is a $y\in X\cap S \setminus D.$  The row of $S$ containing $y$ contains at least $2c\sqrt{N}+1$ additional elements of $X\cap S$ both to the left and to the right of $y,$ and the column of $S$ containing $y$ has at least $2c\sqrt{N}+1$ additional elements of $X\cap S$ both above and below $y.$ Let $q_1,$ $q_2,$ $q_3$ and $q_4$ be elements of $X$ to the right, above, below, and to the left of $y$ so that there  $c\sqrt{N}$ other elements of $X\cap S$ between them and $y,$ and at least $c\sqrt{N}$ additional elements between them and the boundary of $S.$ Translate $S$ so that $y$ is located at $\paren{0,0}$ (ignoring the higher-dimensional coordinates). $q_1, q_2, q_3,$ and $q_4$ form a quadrilateral, so there is a non-obtuse angle at at least one of the vertices. Without loss of generality, suppose the angle at $q_1$ is non-obtuse. Let $q_1=\left(\hat{x},0\right),$ $q_2=\left(0,y_1\right),$ and $q_3=\left(0,y_2\right).$ By construction, $\hat{x}$ and $y_1$ are positive, and $y_2$ is negative. Additionally, let $p_1=\left(x,0\right)$ be an element of $X\cap S$ in the same row as $y$ and $q_1$ that is at least $c\sqrt{N}$ cubes further to the right of $q_1.$ Then $p_1, q_2,$ and $q_3$ form an acute triangle. See Figure~\ref{fig_half}.

$p_1, q_2,$ and $q_3$ form an acute triangle, so $\PH_1(\set{p_1,q_2,q_3})$ consists of a single interval. We will show that the length of this interval is greater than $1+\sqrt{m}$ for sufficiently large $N,$ which will imply that the $\PH_1$ class is a stable class of size greater than one. The length of this interval is equal to the circumradius of the triangle minus half the length of its maximum edge. The circumradius of a triangle with edge lengths $a,b,c$ and area $A$ is $abc/4A$ so the circumradius of $p_1,q_2,q_3$ is
\[\frac{\sqrt{\paren{x^2+y_1^2}\paren{x^2+y_2^2}}\paren{y_1-y_2}}{4\frac{1}{2}\paren{y_1-y_2}x}=\frac{\sqrt{\paren{x^2 + y_1^2} \paren{x^2 + y_2^2}}}{2 x}\,.\]

We consider three cases. In the first case, the edge from $q_2$ to $q_3$ is longest. In this case, the total persistence equals
\[TP_1\paren{x,y_1,y_2}=\frac{\sqrt{\paren{x^2 + y_1^2} \paren{x^2 + y_2^2}}}{2 x}-\frac{y_1-y_2}{2}\,.\]

$q_1,q_2,q_3$ is an acute triangle, so $\hat{x}>-y_1y_2$ and $x>\sqrt{-y_1y_2}+c\sqrt{N}.$ In Appendix~\ref{sec:triangleComputations1}, we show that, subject to the constraints that $c\sqrt{N}\leq y_1 \leq N,$ $-N \leq y_2 \leq -c\sqrt{N},$ and $x>\sqrt{-y_1y_2}+c\sqrt{N}$ (which all hold in this case), 

\[TP_1\paren{x,y_1,y_2}\geq TP_1\paren{c\sqrt{N}+N,N,-N}=\frac{c^2 N}{2 \left(c \sqrt{N}+N\right)}\]

which is an increasing function of $N$ whose limit is $\frac{1}{2}c^2$ as $N$ goes to $\infty.$ As $c>\sqrt{4\sqrt{m}+4}>\sqrt{2\sqrt{m}+2},$ the total persistence will exceed $1+\sqrt{m}$ for sufficiently large $N,$ as desired.

In the second case, the edge between $p$ and $q_1$ is longest. In this case, the total persistence equals
\[TP_2\paren{x,y_1,y_2}=\frac{\sqrt{\paren{x^2 + y_1^2} \paren{x^2 + y_2^2}}}{2 x}-\frac{\sqrt{x^2+y_1^2}}{2}\,.\]
As we show in Appendix~\ref{sec:triangleComputations2}, for sufficiently large $N,$ this is minimized when $x$ is as large as possible and the magnitudes of $y_1$ and $y_2$ are as small as possible. Thus, for sufficiently large $N,$
\[TP_2\paren{x,y_1,y_2}\geq TP_2\paren{N,c\sqrt{N},-c\sqrt{N}}= \frac{1}{2} \paren{c^2 + N - \sqrt{c^2 N + N^2}}\]
which is a decreasing function of $N$ with limiting value $\frac{1}{4}c^2.$ Because $c>\sqrt{4\sqrt{m}+4}$, the total persistence is always greater than $1+\sqrt{m},$ as desired. The argument in the third case is identical to this one.
\end{proof}
 
The previous result, together with Proposition~\ref{prop:kappaprop}, completes the proofs of Theorems~\ref{mainTheorem} and~\ref{thmLower}.

\section{Results for the Rips Complex}
\label{sec:rips}
In this section, we first  construct an example where for which $\text{dim}_{\widetilde{\PH}_1}\paren{X}=2$ but $\text{dim}_{\text{box}}\paren{X}=1,$ then we show that this cannot occur in Euclidean space by proving Theorems~\ref{ripsTheorem} and~\ref{theorem_rips_count}.

\subsection{An Example}
\label{RipsExample}
Here, we prove Proposition~\ref{prop:ripsExample} by constructing an example of a metric space not embeddable in any Euclidean space whose $\widetilde{\PH}_1$ dimension equals 2 but whose upper box dimension equals 1. This contrasts with the $\widetilde{\PH}_0$ dimension, which was proven by Kozma, Lotker, and Stupp~\cite{Kozma} to agree with the upper box dimension for all metric spaces. The rough idea of the construction is to build a space that includes a copies of the complete bipartite graph on $2^n$ vertices at different scales.

\begin{proof}[Proof of Proposition~\ref{prop:ripsExample}]
We will construct $X$ as the union of finite point sets $X_n.$  For each $n\in\mathbb{N}$ let  $X_n$ be the set consisting of $2^{n+1}$ points $x^n_1, x^n_2,\ldots, x^n_{2^n}$ and $y^n_1, y^n_2,\ldots, y^n_{2^n},$ and let
\[d\paren{x^n_i,y^n_i}=\frac{1}{2^{n+1}}\hspace{.5in}d\paren{x^n_i,x^n_j}=\paren{1-\delta_{i,j}}\frac{1}{2^n}\hspace{.5in}d\paren{y^n_i,y^n_j}=\paren{1-\delta_{i,j}}\frac{1}{2^n}\,.\]
Let $X=\cup_{i=0}^\infty X_i.$ if $i\neq j,$ $z\in X_i$ and $w\in X_j,$ set
\begin{equation}
\label{eqn_ripsexample}
d\paren{z,w}=\frac{1}{2^{\text{min}\paren{i,j}}}\,.
\end{equation}

We will show the upper box dimension of $X$ equals one by computing $N_{\delta}\paren{X},$ the maximum number of disjoint closed balls of radius $\delta$ centered at points of $X.$ 

Suppose $2^{-n-1}<\delta<2^{-n}$ and that the balls centered at $\textbf{x}$ are a maximal collection of disjoint balls of radius $\delta$ centered at points of $X.$ Let $Y_n=\cup_{j<n}X_j$ and $Z_n=\cup_{j>n}X_j.$ If $y\in Y_n,$ $B_{\delta}\paren{y}$ contains every point of $Y_n,$ but no point of $X\setminus Y_n.$ Therefore, $\textbf{x}$ contains one point of $Y_n.$ Also, for $1\leq i \leq 2^n,$ $B_{\delta}\paren{y^n_i}$ contains $y^n_i$ and the $2^n$ points $x^n_1,\ldots,x^n_{2^n}.$ By symmetry,  $B_{\delta}\paren{x^n_i}$ contains $x^n_i$ and the $2^n$ points $y^n_1,\ldots,y^n_{2^n}.$ Therefore, any two $\delta$-balls centered at points of $X_n$ intersect and $\textbf{x}$  contains either one point $y^n_i$ or one point $x^n_i.$ Finally, if $z\in Z_n,$  $B_{\delta}\paren{z}\cap X=z.$ Therefore, $\textbf{x}$ must contain all of the points of $Z_n.$ It follows that 
\begin{align*}
N_{\delta}\paren{X}=&\;\;\abs{\textbf{x}}\\
=&\;\; 1+2^n+\abs{Z_n}\\
=&\;\; 1+\sum_{j=0}^{n-1}\abs{X_n}\\
=&\;\; 1+\sum_{j=0}^{n-1}2^{n+1}\\
=&\;\; \ 2^{n+1} -1 \\
=& \;\;  2^{-\ceil{\frac{\log\paren{\delta}+1}{\log\paren{2}}}}-1\,.
\end{align*}
 Therefore 
\begin{align*}
\text{dim}_{\text{box}}\paren{X}=&\;\;\limsup_{\delta\rightarrow 0}{\frac{\log\paren{N_\delta\paren{X}}}{-\log\paren{\delta}}}\\
= &\;\;\lim_{\delta \rightarrow 0}\frac{\log\paren{2^{-\ceil{\frac{\log\paren{\delta}}{\log\paren{2}}}+1}-1}}{-\log\paren{\delta}}\\
=&\;\; \lim_{\delta \rightarrow 0} \frac{-\ceil{\frac{\log\paren{\delta}}{\log\paren{2}}}\log\paren{2}}{-\log\paren{\delta}}\\
= & \;\; 1\,.
\end{align*}

On the other hand, consider the Rips complex  $\mathcal{R}_{\epsilon}\paren{X_n}$ on $X_n.$  There are three regimes. If $\epsilon<\frac{1}{2^{n+1}},$ then $\mathcal{R}_{\epsilon}\paren{X_n}$ consists of $2^{n+1}$ $0$-cells. When $\epsilon=\frac{1}{2^{n+1}},$ the $2^{2n}$ edges between the points $\set{x_i^n}_{i=1^{2^n}}$ and $\set{y_i^n}_{i=1^{2^n}}$ enter the complex. The complex remains unchanged until $\epsilon=\frac{1}{2^n},$ when all possible simplices enter the complex and it becomes contractible. As such, all intervals of $\widetilde{PH}_1\paren{X_n}$ are born at $\epsilon=\frac{1}{2^{n+1}}$ and die at $\epsilon=\frac{1}{2^n}.$ By an Euler characteristic argument, there are $2^{2n}-2^{n+1}+1$ such intervals. Therefore, for $\alpha>0$ 
\begin{align*}
\widetilde{E}^1_\alpha\paren{X_n}=&\;\;\paren{2^{2n}-2^{n+1}+1}\paren{\frac{1}{2^n}-\frac{1}{2^{n+1}}}^\alpha\\
=&\;\;\paren{2^{2n}-2^{n+1}+1}\frac{1}{2^{\alpha n + \alpha}}\\
=&\;\;2^{-\alpha}\paren{2^{\paren{2-\alpha}n}-2^{\paren{1-\alpha}n+1}+2^{-\alpha n}}\\
\approx&\;\;2^{-\alpha}2^{\paren{2-\alpha}n}\,,
\end{align*}
which is which is unbounded as $n\rightarrow\infty$ if $\alpha<2.$ Thus $\text{dim}_{\widetilde{\PH}_1}(X)\geq 2.$ 

The argument that $\text{dim}_{\widetilde{\PH}_1}(X)\leq 2$ is similar but slightly trickier. Let $\alpha>2,$ $\textbf{x}\subseteq X,$ and $\textbf{x}_n=\textbf{x}\cap X_n.$ An argument nearly identical to the previous one shows that  $\widetilde{PH}_1\paren{\textbf{x}_n}$ consists of at most $2^{2n}-2^{n+1}+1$ intervals of the form $\paren{\frac{1}{2^{n+1}},\frac{1}{2^n}},$ and
\begin{equation}
\label{eqn_ripsexample3}
\widetilde{E}^1_\alpha\paren{\textbf{x}_n}< 2^{-\alpha}\paren{2^{\paren{2-\alpha}n}+2^{-\alpha n}}\,.
\end{equation}

We will show that we can write 
\[\widetilde{E}^1_\alpha\paren{\textbf{x}}=\sum_{n=0}^\infty \widetilde{E}_\alpha^1 \paren{\textbf{x}_n}\,.\]
Let $n\in\mathbb{N}$ and define
\[K_{i,n}\paren{Y}=\set{\paren{b,d}\in \widetilde{\PH}_i\paren{Y}:\frac{1}{2^{n+1}} \leq b <  \frac{1}{2^{n}}}\,.\]

Also, let 
 \[\frac{1}{2^{n+1}}\leq \epsilon < \frac{1}{2^n},\quad \textbf{y}_n=\cup_{j<n}\textbf{x}_j,\text{ and }\textbf{z}_n=\cup_{j>n}\textbf{x}_j\,.\]
By equation~\ref{eqn_ripsexample},  $\mathcal{R}_{\epsilon}\paren{\textbf{x}}$ contains no edges between points of $\textbf{y}_n$ and points of $\textbf{x}_n.$ Furthermore, no point of $\textbf{z}_n$ is contained in an edge. Therefore, 
\[\mathcal{R}_{\epsilon}\paren{\textbf{x}}=\mathcal{R}_{\epsilon}\paren{\textbf{x}_n}\sqcup \mathcal{R}_{\epsilon}\paren{\textbf{y}_n}\sqcup \mathcal{R}_{\epsilon}\paren{\textbf{z}_n}\,,\]
where $\sqcup$ denotes a disjoint union. $\mathcal{R}_{\epsilon}\paren{\textbf{y}_n}$ is contractible because it contains all possible simplices, so $H_1\paren{\mathcal{R}_{\epsilon}\paren{\textbf{y}_n}}=0.$ $H_1\paren{\mathcal{R}_{\epsilon}\paren{\textbf{z}_n}}$ is also trivial because $\mathcal{R}_{\epsilon}\paren{\textbf{z}_n}$ contains only zero cells. Therefore, $H_1\paren{\mathcal{R}_{\epsilon}\paren{\textbf{x}}}=H_1\paren{\mathcal{R}_{\epsilon}\paren{\textbf{x}_n}}.$ Also, if $\delta>\frac{1}{2^n}$ then $\mathcal{R}_{\delta}\paren{\textbf{x}_n}$ is contractible so the inclusion map  $H_1\paren{\mathcal{R}_{\epsilon}\paren{\textbf{x}}}\rightarrow H_1\paren{\mathcal{R}_{\delta}\paren{\textbf{x}}}$ is trivial. It follows that

\begin{equation}
\label{eqn_ripsexample2}
K_{1,n}\paren{\textbf{x}}=K_{1,n}\paren{\textbf{x}_n}=\widehat{\PH}_1\paren{\textbf{x}_n}\,.
\end{equation}

Let $\alpha>2$ and compute
\begin{align*}
\widetilde{E}^1_\alpha\paren{\textbf{x}}=&\;\; \sum_{n\in\mathbb{N}}\sum_{I\in K_{1,n}\paren{\textbf{x}}}\abs{I}^\alpha\\
=&\;\; \sum_{n\in\mathbb{N}} \sum_{I\in \widehat{PH}_1\paren{\textbf{x}_n}}\abs{I}^\alpha &&\text{by Eqn~\ref{eqn_ripsexample2}}\\
=&\;\;\sum_{n\in\mathbb{N}} \widetilde{E}^1_\alpha\paren{\textbf{x}_n}\\
\leq &\;\;  \sum_{n\in\mathbb{N}} 2^{-\alpha}\paren{2^{\paren{2-\alpha}n}+2^{-\alpha n}} &&\text{by Eqn~\ref{eqn_ripsexample3}} \\
= &\;\;  2^{-\alpha}\paren{\frac{1}{1-2^{\paren{2-\alpha}}}+\frac{1}{1+2^{-\alpha}}} && \text{because $\alpha>2\,.$}
\end{align*}
Therefore $\widetilde{E}^1_\alpha\paren{\textbf{x}}$ is uniformly bounded for all $\textbf{x}\subset X$ when $\alpha>2,$ and $\text{dim}_{\widetilde{\PH}_1}\paren{X}= 2.$ 
\end{proof}

\subsection{A Bound for $\text{dim}_{\widetilde{\PH}_1}$}
\label{sec_rips}
We prove an upper bound for the $\widetilde{\PH}_1$-dimension of a subset of $\R^m$ that does not depend on the ambient Euclidean dimension. Note that the proof of Proposition~\ref{prop1} (and the preceding lemmas) works for persistent homology defined in terms of either the \v{C}ech complex or the Rips complex.

The following argument is a modified version of the proof of Theorem 3.1 in Goff~\cite{Goff2011}, which bounds the (non-persistent) first Betti number of a Rips complex of a finite point set in $\R^m$ in terms of the kissing number $K_m$ of $\R^m.$ Recall that $K_m$ is the maximum number of disjoint unit spheres that can be tangent to a shared unit sphere.  For example $K_1=2, K_2=6,$ and $K_3=12.$ There is an expansive literature on upper and lower bounds for $K_m$~\cite{2004pfender}, but here will we just need that $K_m$ is finite.

 Let $\mathcal{R}_{\epsilon}\paren{X}$ be the Rips complex of $X$ at parameter $\epsilon.$ For a $0$-cell $x$ let $S_{\epsilon}\paren{x}$ denote the star of $x$
\[S_{\epsilon}\paren{x}=\set{\sigma\in \mathcal{R}_\epsilon\paren{X}: x\in \sigma}\,,\]
and $L_{\epsilon\paren{x}}$ denote the link of $x$ 
\[L_{\epsilon}(x)=\set{\paren{z_1,\ldots z_k}\in \mathcal{R}_\epsilon\paren{X}:\paren{x,z_1,\ldots z_k} \in   \mathcal{R}_\epsilon\paren{X}}\,.\]
$S_{\epsilon\paren{x}}$ is contractible.

\begin{figure}
\center
\includegraphics[width=200pt]{%
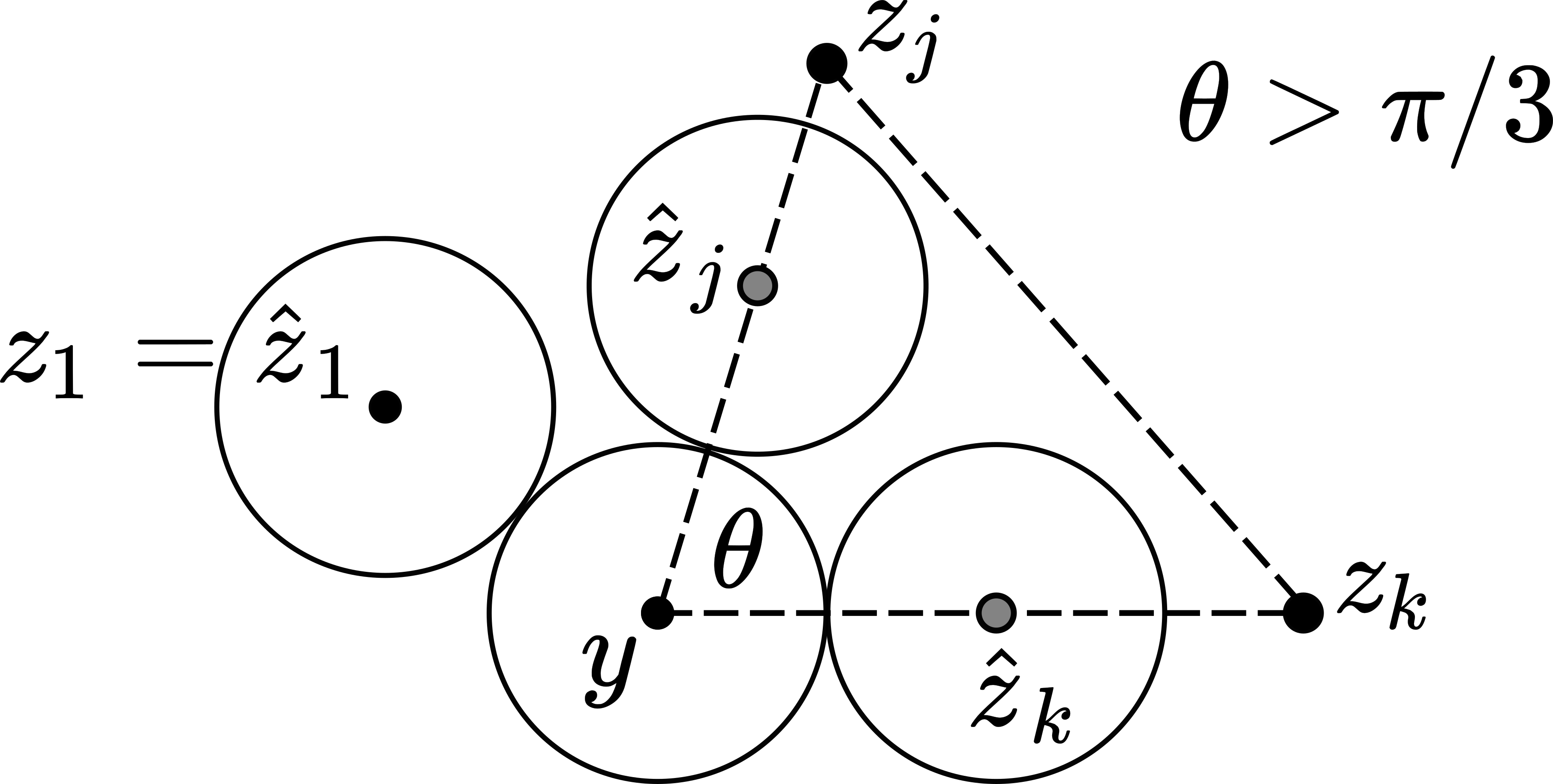}
\caption{\label{fig:triangleRips} The setup in the proof of Lemma~\ref{lemma_kissing}, where $z_1$ is the closest point of $\textbf{x}\setminus y$ to $y.$}
\end{figure}

\begin{Lemma}
\label{lemma_kissing}
Let $\textbf{x}$ be a finite subset of $\R^m,$ and $y\in \textbf{x},$ and let $\set{L_{\epsilon}\paren{y}}_{\epsilon\in\R^+}$ be the filtration of the links of $y$ in $\mathcal{R}_{\epsilon}(\textbf{x}).$ Then
\[\abs{\PH_0\paren{\set{L_{\epsilon}\paren{y}}_{\epsilon\in\R^+}}}\leq K_m-1\,.\]
\end{Lemma}
\begin{proof}

If necessary, perturb the points of $\textbf{x}$ so all pairwise distances are distinct without reducing the number of $\PH_0$ bars (by Corollary~\ref{stableCorollary}, this can be done by perturbing the points by a distance less than $\delta/2,$ where $\delta$ is the length of the shortest $\PH_0$ bar). Let the $\PH_0$ intervals of the filtration be $\set{\paren{b_j,d_j}}_{j=1}^k.$ For each $j,$ the number of components of $L_{\epsilon}\paren{y}$ increases at $\epsilon=b_j.$ The only way a new component can be added to the link of $y$ at $\epsilon=b_j$ is if there is a point $z_j$ so that  $b_j=d\paren{z_j,y}.$ Let $\set{z_j}_{j=1}^k$ be the collection of these points.
  
Let $\epsilon=d\paren{y,\textbf{x}\setminus y}/2.$ For each $z_j,$ let $\hat{z}_j$ be the point on the line segment from $y$ to $z_j$ with $d\paren{y,\hat{z}_j}=2\epsilon.$ By construction, the balls of radius $\epsilon$ centered at the points $\hat{z}$ are tangent to the ball of radius $\epsilon$ centered at $y.$ We will show that they are disjoint by performing another computation related to the non-hatted points $\set{z_j}.$

Let $z_j$ and $z_k$ be distinct and assume without loss of generality that $d\paren{z_k,y}>d\paren{z_j,y}.$ See Figure~\ref{fig:triangleRips}. Then $z_j$ and $z_k$ are contained in in $L_{d\paren{z_k,y}}\paren{y},$ but they must be in distinct components (if not, the number of components of the link would not increase at $\epsilon=d\paren{z_k,y}$). In particular,
\[d\paren{z_j,z_k}>d\paren{z_k,y}>d\paren{z_j,y}\,.\]
Consider the triangle $T_0$ formed by $z_j,z_k,$ and $y.$ The previous equation shows the edge between $z_j$ and $z_k$ is the longest edge of $T_0$, so the angle opposite to it must be greater than $\pi/3.$ In other words, the angle between the vectors $z_j-y$ and $z_k-y$ is greater than $\pi/3.$  By construction, the angle between the vectors $\hat{z}_i-y$ and $\hat{z}_j-y$ equals the angle between the vectors $z_j-y$ and $z_k-y.$ Furthermore, these vectors have the same length, $2\epsilon.$ Therefore, the triangle formed by $\hat{z}_j,\hat{z}_k$ and $y$ is an isosceles triangle so that the angle opposite the edge $\paren{\hat{z}_j,\hat{z}_k}$ is greater than $\pi/3.$ It follows that this is the longest edge of the triangle and $d\paren{\hat{z}_j,\hat{z}_k}>2\epsilon.$  

Therefore, the balls $\set{B_\epsilon\paren{\hat{z}_j}}$ are disjoint and tangent to $B_{\epsilon}\paren{y},$ so there are at most $K_m$ of them. Thus
\[ \abs{\PH_0\paren{\set{L_{\epsilon}\paren{y}}_{\epsilon\in\R^+}}}=\abs{\set{z_j}}-1=\abs{\set{B_\epsilon\paren{\hat{z}_j}}}\leq K_m-1\,,\]
where the ``$-1$'' comes from the fact that are using reduced homology.
\end{proof}

\begin{Lemma}
\label{lemma:changeByone}
Let $\textbf{x}$ be a finite metric space so that all pairwise distances are distinct, and let $\epsilon_1=0<\epsilon_2<\ldots<\epsilon_p$ be the values of $\epsilon$ at which a simplex is added to the Rips complex of $\textbf{x}.$
 Then 
\s{\abs{\widetilde{PH}_1\paren{\textbf{x}}}=\abs{\set{j:\text{dim}\:H_1\paren{\mathcal{R}_{\epsilon_{j}}\left(\textbf{x}\right)}=\text{dim}\:H_1\paren{\mathcal{R}_{\epsilon_{j-1}}\left(\textbf{x}\right)}+1}}\,.}
\end{Lemma}
\begin{proof}
We will show that there is bijection between the birth times $\PH_1$ intervals and filtration values $\epsilon_j$ so that the rank of the first homology increases by one. If 
\[\text{dim}\:H_1\paren{\mathcal{R}_{\epsilon_{j}}\left(\textbf{x}\right)}=\text{dim}\:H_1\paren{\mathcal{R}_{\epsilon_{j-1}}\left(\textbf{x}\right)}+1\]
then there is persistent homology interval born at $\epsilon_j,$ with death time at least $\epsilon_{j+1}.$ \footnote{Note that a similar statement is false for general Rips complexes on infinite point sets; there can be ``ephermal'' homology classes that are born and die at the same value of $\epsilon$ and do not correspond to any persistent homology interval.} Therefore
\s{\abs{\widetilde{PH}_1\paren{\textbf{x}}}\geq \abs{\set{j:\text{dim}\:H_1\paren{\mathcal{R}_{\epsilon_{j}}\left(\textbf{x}\right)}=\text{dim}\:H_1\paren{\mathcal{R}_{\epsilon_{j-1}}\left(\textbf{x}\right)}+1}}\,.}

Suppose a $\PH_1$ interval is born at $\epsilon_j.$  Note that, by the definition of the Rips complex, the $\epsilon_k$'s are precisely the pairwise distances between the points of $\textbf{x}.$  It follows that exactly one edge enters the complex at $\epsilon_j:$ the unique edge $(z_j,y_j)$ so that $d\paren{z_j,y_j}=\epsilon_j.$ The Euler characteristic of the $1$-skeleton of the Rips complex increases exactly by one, so
\[\text{dim}\:H_1\paren{\mathcal{R}_{\epsilon_{j}}\left(\textbf{x}\right)}\leq \text{dim}\:H_1\paren{\mathcal{R}_{\epsilon_{j-1}}\left(\textbf{x}\right)}+1\,.\]

A $\PH_1$ interval is born at $\epsilon_j$ so there is a cycle $\sigma\in H_1\paren{\mathcal{R}_{\epsilon_{j}}\left(\textbf{x}\right)}$ that is not homologous to a cycle in $H_1\paren{\mathcal{R}_{\epsilon_{j-1}}\left(\textbf{x}\right)}.$ $\sigma$ must contain the edge $(y_j,z_j).$ We will show that the contribution of $\sigma$ to the rank of $H_1\paren{\mathcal{R}_{\epsilon_{j}}\left(\textbf{x}\right)}$ cannot be canceled by the simultaneous death of another $\PH_1$ interval.  By way of contradiction, assume that another $\PH_1$ interval has death time equal to $\epsilon_j,$ so a $2$-simplex also enters the Rips complex at $\epsilon_j.$ Then it must be of the form $\paren{y_j,z_j,w_j}$ for some $w_j\in \textbf{x}.$ Then $\sigma$ is homologous to the cycle obtained by replacing the edge $(y_j,z_j)$ with the edges $(y_j,z_j)$ and $(z_j,w_j).$ This is a contradiction, so every cycle in  $H_1\paren{\mathcal{R}_{\epsilon_{j-1}}\left(\textbf{x}\right)}$ is a cycle in $H_1\paren{\mathcal{R}_{\epsilon_{j}}\left(\textbf{x}\right)}$ and
\[\text{dim}\:H_1\paren{\mathcal{R}_{\epsilon_{j}}\left(\textbf{x}\right)}=\text{dim}\:H_1\paren{\mathcal{R}_{\epsilon_{j-1}}\left(\textbf{x}\right)}+1\,.\]

\end{proof}

\begin{proof}[Proof of Theorem~\ref{theorem_rips_count}]
We will proceed by  induction on the size of $\textbf{x}.$ Let $\delta$ be the length of the shortest interval of $\widetilde{PH}_1\paren{\textbf{x}}.$ Perturb the points of $\textbf{x}$ by an amount less than $\delta/2$ so that the lengths of all of the pairwise distances between points of $\textbf{x}$ are distinct. The stability of the bottleneck distance implies that no bars of $\widetilde{\PH}_1\paren{\textbf{x}}$ are destroyed by this perturbation, so an upper bound for the perturbed set implies one for the original. 

For $\epsilon\geq 0$ and $y\in\textbf{x},$ we have that 
\[\mathcal{R}_{\epsilon}\paren{\textbf{x}}=\mathcal{R}_{\epsilon}\left(\textbf{x}\setminus y\right)\cup S_{\epsilon}\left(y\right)\,\,\]
and 
\[\mathcal{R}_{\epsilon}\left(\textbf{x}\setminus y\right)\cap S_{\epsilon}\left(y\right)=L_{\epsilon}\paren{x}\,.\]
As such, there is a Mayer--Vietoris sequence.
\[\ldots\rightarrow H_1\paren{L_{\epsilon}\left(y\right)}\rightarrow H_1\paren{\mathcal{R}_{\epsilon}\paren{\textbf{x}\setminus y}}\oplus H_1\paren{ S_{\epsilon}\left(y\right)} \rightarrow  H_1\paren{\mathcal{R}_{\epsilon}\paren{\textbf{x}}}\rightarrow H_0\paren{L_{\epsilon}\left(y\right)}\rightarrow\ldots\]

Let $\epsilon_1=0<\epsilon_2<\ldots<\epsilon_p$ be the values of $\epsilon$ at which a simplex is added to the Rips complex of $\textbf{x}.$ The Mayer--Vietoris sequences for each $\epsilon_j$ fit into the following commutative diagram, where vertical maps are inclusions and we have suppressed $H_1\paren{S_{\epsilon}\paren{x}}=0$ in the left column.

\begin{center}
\begin{tikzcd}
H_1\paren{\mathcal{R}_{\epsilon_{i-1}}\left(\textbf{x}\setminus y\right)} \arrow[r,"\alpha_{i-1}"] \arrow[d] & H_1\paren{\mathcal{R}_{\epsilon_{i-1}}\left(\textbf{x}\right)} \arrow[r,"\partial_{i-1}"] \arrow[d] & H_0\paren{L_{\epsilon_{i-1}}\left(y\right)}\arrow[d,"\xi_i"] \\
H_1\paren{\mathcal{R}_{\epsilon_{i}}\left(\textbf{x}\setminus y\right)} \arrow[r,"\alpha_i"] \arrow[d] & H_1\paren{\mathcal{R}_{\epsilon_{i}}\left(\textbf{x}\right)} \arrow[r,"\partial_i"] \arrow[d] & H_0\paren{L_{\epsilon_{i}}\left(y\right)}\arrow[d,"\xi_{i+1}"] \\
H_1\paren{\mathcal{R}_{\epsilon_{i+1}}\left(\textbf{x}\setminus y\right)} \arrow[r,"\alpha_{i+1}"] & H_1\paren{\mathcal{R}_{\epsilon_{i+1}}\left(\textbf{x}\right)} \arrow[r,"\partial_{i+1}"] & H_0\paren{L_{\epsilon_{i+1}}\left(y\right)}
\end{tikzcd}
\end{center}

Recall from the proof of the previous lemma that $H_1\paren{\mathcal{R}_{\epsilon_{i}}\left(\textbf{x}\right)}$ and $H_1\paren{\mathcal{R}_{\epsilon_{i}}\left(\textbf{x}\setminus y\right)}$ can increase at most by one from row to row, and that the number of times the rank increases is equal to the number of $\PH_1$ intervals of the column. 

Homology is taken with field coefficients so
\s{H_1\paren{\mathcal{R}_{\epsilon_{i}}\left(\textbf{x}\right)}\cong\text{im}\,\alpha_{i}\oplus\text{im}\,\partial_{i}\,.} 
Therefore, if 
\[\text{dim}\:H_1\paren{\mathcal{R}_{\epsilon_{i}}\left(\textbf{x}\right)}=\text{dim}\:H_1\paren{\mathcal{R}_{\epsilon_{i-1}}\left(\textbf{x}\right)}+1\] the dimension of either $\text{im}\,\alpha_{i}$ or $\text{im}\,\partial_{i}$ must also have increased. In the former case, commutativity of the diagram (combined with the previous lemma) implies that  
\s{\text{dim}\:H_1\paren{\mathcal{R}_{\epsilon_{i}}\left(\textbf{x}\setminus y\right)}= \text{dim}\:H_1\paren{\mathcal{R}_{\epsilon_{i-1}}\left(\textbf{x}\setminus y\right)}+1\,,}
so an interval of $\widetilde{PH}_1\paren{\textbf{x}\setminus y}$ must also be born at time $\epsilon_{i}.$ In latter case, commutativity of the diagram implies that there a $\sigma_i\in H_0\paren{L_{\epsilon_{i}}\paren{y}}$ so that
\begin{equation}
\label{eqn_temp_rips}
\sigma_i\in \text{im}\,\partial_{i} \text {  and  } \sigma_i \notin \text{im}\, \xi_{i}\circ \xi_{i-1}\ldots \circ \xi_{k+1}\circ\partial_{k} \;\forall\; 1\leq k < i \,.
\end{equation}
Let 
\[\hat{i}=\max\set{j: 1\leq j\leq i\text{ and }\sigma_i \notin \text{im}\, \xi_{i}\circ \xi_{i-1}\ldots\circ \xi_j}.\]
(Note that we are using reduced homology, so the row of diagram with $i=1$ is all zero and $\xi_1=0.$). We may assume that $\sigma_i$ was chosen to give the minimal possible value of $\hat{i},$ subject to Equation~\ref{eqn_temp_rips}. By construction, there is an element of $H_1\paren{L_{\epsilon_{\hat{i}}}}$  that is not in the image of $\xi_{\hat{i}}$ (the one that maps to $\sigma_i$), so there is a $\PH_0$ interval $I_i$ of the link filtration of $y$ with birth time equal to $\epsilon_{\hat{i}}.$ This interval is unique, by the same reasoning as in Lemma~\ref{lemma_kissing}. We will show that $i\mapsto I_{i}$ gives an injection
\s{\set{i:\text{dim}\:\text{im}\,\partial_{i}>\text{dim}\:\text{im}\,\partial_{i-1}} \rightarrow\widetilde{PH}_0\paren{\set{L_{\epsilon}\paren{x}}}\,.}

By way of contradiction, suppose that  that 
\[\epsilon_j, \epsilon_l\in \set{i:\text{dim}\:\text{im}\,\partial_{i}>\text{dim}\:\text{im}\,\partial_{i-1}}\]
 and $j<l$  but $I_j=I_l$ (so $\hat{j}=\hat{l}$). 
 Choose $\hat{\sigma}_j,\hat{\sigma}_l\in H_0\paren{L_{\epsilon_{\hat{l}}}\paren{y}}$  so that
\[\xi_j \circ \ldots \xi_{\hat{l}}\paren{\hat{\sigma}_j}=\sigma_j\text{ and }\xi_l \circ \ldots \xi_{\hat{l}}\paren{\hat{\sigma}_l}=\sigma_l\,.\]
Let $z$ be the unique point of $\textbf{x}$ so that $z\in L_{\epsilon_{\hat{l}}}\paren{y}$ but $z\notin L_{\epsilon_{\hat{l}}-1}\paren{y}.$ By the minimality of $\hat{j}$ and $\hat{l},$ there are $a_j,a_l$ in the coefficient field so that $a_j,a_l\neq 0$ and 
\[\hat{\sigma}_j-a_j z, \hat{\sigma}_l-a_l z \in \text{im}\,\xi_{\hat{l}}\,.\]
$j<l$ and $\frac{a_l}{a_j}\sigma_j\in\text{im}\:\,\partial_{j},$ so it follows from commutativity of the diagram that $\sigma_l-\frac{a_l}{a_j}\sigma_j$ must also satisfy Equation~\ref{eqn_temp_rips} for $i=l,$ but
\[\sigma_l-\frac{a_l}{a_j}\sigma_j=\xi_l \circ \ldots \circ \xi_{\hat{l}}\paren{\hat{\sigma}_l-\frac{a_l}{a_j}\hat{\sigma_j}}\]
and $\hat{\sigma}_l-\frac{a_l}{a_j}\hat{\sigma_j} \in \text{im}\,\xi_{\hat{l}},$ which contradicts the minimality of $\hat{l}.$ Therefore 

\s{\abs{\set{i:\text{dim}\:\text{im}\,\partial_{i}>\text{dim}\:\paren{\text{im}\,\partial_{i-1}}}}\leq \abs{\widetilde{PH}_0\paren{\set{L_{\epsilon}\paren{x}}}}\,.}
 
Thus, by applying the previous lemma,
\begin{align*}
\abs{\widetilde{PH}_1\paren{\textbf{x}}} = & \;\; \abs{\set{i:\text{dim}\:H_1\paren{\mathcal{R}_{\epsilon_{i+1}}\left(\textbf{x}\right)}=\text{dim}\:H_1\paren{\mathcal{R}_{\epsilon_{i}}\left(\textbf{x}\right)}+1}}\\
=&\;\;\abs{\set{i:\text{dim}\:\text{im}\,\partial_{i}>\text{dim}\:\text{im}\,\partial_{i-1}}}+ \abs{\set{i:\text{dim}\:\text{im}\,\alpha_{i}>\text{dim}\alpha_{i-1}}}\\
 \leq & \;\;\abs{\widetilde{PH}_1\paren{\textbf{x}\setminus y}} + \abs{\widetilde{PH}_0\paren{\set{L_{\epsilon}\paren{x}}_{\epsilon\in\R^+}}}\\
\leq &\;\; \abs{\widetilde{PH}_1\paren{\textbf{x}\setminus y}}+K_{m}-1 &&\text{by Lemma~\ref{lemma_kissing}}\\
\leq & \;\; \paren{K_m-1} \abs{x} &&\text{by induction}\,.
\end{align*}

\end{proof}

This, together with Proposition~\ref{prop1}, implies Theorem~\ref{ripsTheorem}.

\section{Conclusion}
We have taken the first steps toward answering Question~\ref{MainQuestion}, which asks for hypotheses under which the $\PH_i$-dimension of a bounded subset of $\R^m$ equals its upper box dimension. However, the question remains open for all cases with $m>2$ and $i>0.$ We suspect that even the $n=2,i=1$ case can be improved, to include sets whose upper box dimension is between $1$ and $1.5.$ 

Another interesting question is whether similar results can be shown for a probabilistic version of the $\PH_i$-dimension. That is, if $\textbf{x}$ is a finite point collection drawn from a probability measure on $\R^m,$ can the expectation of $E_i\paren{\textbf{x}}$ be controlled in terms a classically defined fractal dimension for probability measures?  This question would perhaps be more interesting for applications, which usually deal with random point collections rather than extremal ones. However, proving a lower bound is already difficult in the extremal case.

\section{Acknowledgments}
The author would like to thank Henry Adams, Tamal Dey, Herbert Edelsbrunner, Matthew Kahle, Chris Peterson, and Shmuel Weinberger for useful discussions. In particular, he would like to thank Herbert Edelsbrunner for suggesting the example in Section~\ref{sec:Arcs} and the reference~\cite{2010CohenSteiner}. He would also like to thank the members of the Pattern Analysis Laboratory at Colorado State University, whose computational experiments involving a similar dimension inspired the ideas in this paper~\cite{2019adams}. 

Funding for this research was provided by a NSF Mathematical Sciences Postdoctoral Research Fellowship under award number DMS-1606259.

\appendix
\section{Properties of $\PH$ complexity}
\label{sec:compWithMeas}

\begin{figure}
\center
\includegraphics[width=.3\textwidth]{%
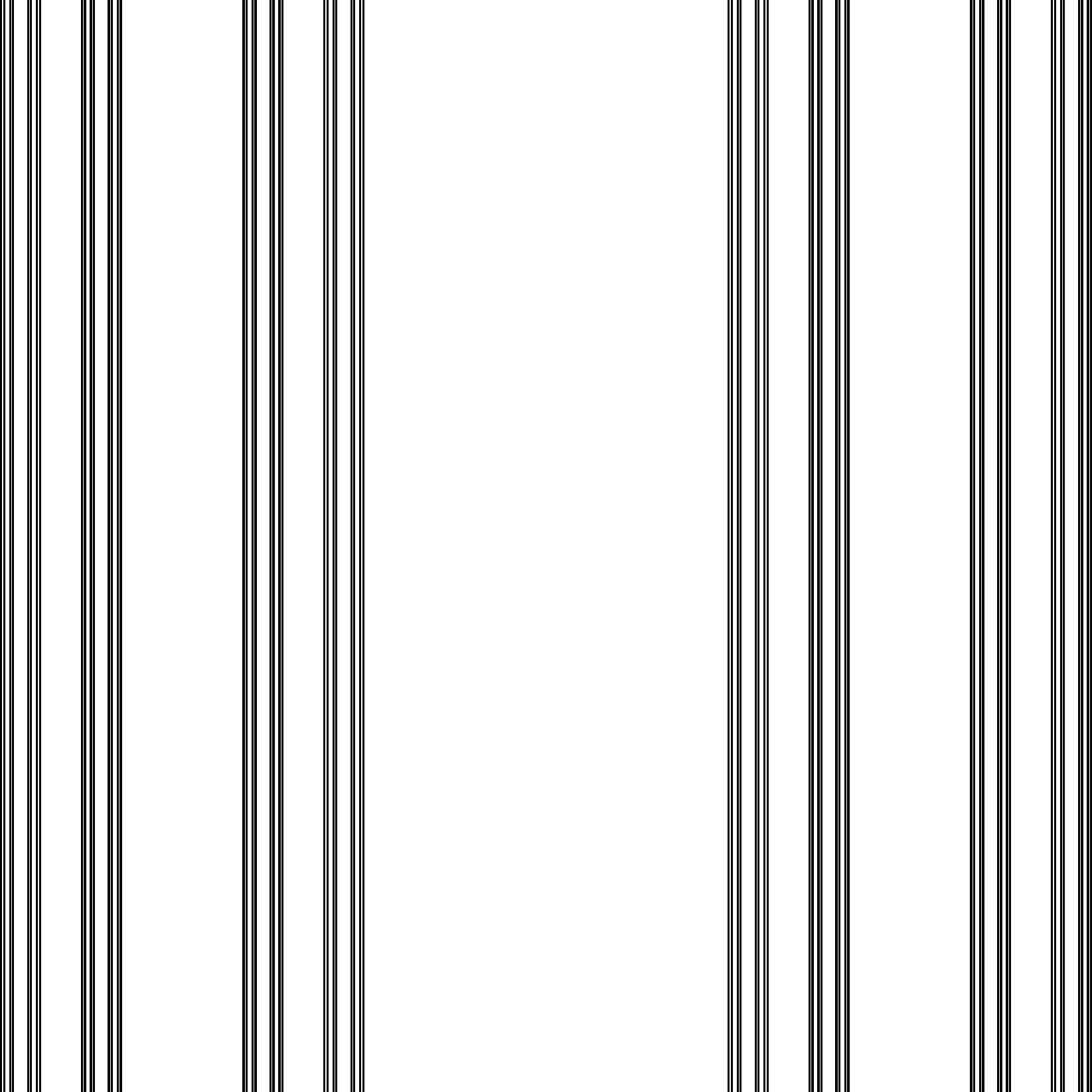}
\caption{\label{fig_CantorCrossInterval} The Cantor set cross an interval. If $C$ is the cantor set and $I$ is an interval, $\text{dim}_{\widehat{\PH}}^0\paren{C\times I}=\text{dim}_{\widehat{\PH}}^0\paren{C}=\log_3{2}$ and $\text{dim}_{\widehat{\PH}}^1\paren{C\times I}=0,$ but $\text{dim}_{\PH}^0\paren{C\times I}=\text{dim}_{\PH}^1\paren{C\times I}=\text{dim}_{\text{box}}\paren{C\times I}=1+\log_3{2}.$}
\end{figure}

In previous work with MacPherson~\cite{2012macpherson}, we defined another notion of persistent homology dimension for a compact subset of a metric space based on the persistent homology of the subset itself rather than of the finite point sets contained within it. Here, we denote it by $\text{dim}_{\widehat{\PH}}^i\paren{X}.$ $\text{dim}_{\widehat{\PH}}^i\paren{X}$ measures  the complexity of a shape rather than any classicly defined fractal dimension, so we refer to it as  the``$\PH_i$ complexity'' of $X.$ See Figure~\ref{fig_CantorCrossInterval}.

\begin{Definition}
The $i$-th dimensional persistent homology complexity of compact subset of a metric space is
\[\text{dim}_{\widehat{\PH}}^i\paren{X}=\sup\set{c:\lim_{x\rightarrow 0}x^c I_{i,\epsilon}\paren{X}=\infty}\,.\]
\end{Definition}

We show that there is an equivalent definition that looks more similar to that of $\text{dim}_{\PH}^i.$

\begin{Proposition}
Let $X$ be a compact subset of a metric space. Then
\[\text{dim}_{\widehat{\PH}}^i\paren{X}=\inf\set{\alpha : E_\alpha^i\paren{X} < \infty}\,.\]
\end{Proposition}
The proposition is an immediate consequence of the following lemma.
\begin{Lemma}
Suppose $Y\subset\R^+$ has the property that $\set{y\in Y: y>\epsilon}$ is finite for all $\epsilon>0.$ If
\[F\paren{\epsilon}=\abs{\set{y\in Y: y>\epsilon}}\]
then 
\[\inf\set{\alpha:\sum_{y\in Y} y^\alpha<\infty}=\sup\set{c:\lim_{\epsilon\rightarrow 0}\epsilon^c F\paren{\epsilon}=\infty}\,.\]
\end{Lemma}
\begin{proof}
Let  $d_1=\inf\set{\alpha:\sum_{y\in Y} y^\alpha<\infty}$ and $d_2=\sup\set{c:\lim_{\epsilon\rightarrow 0}\epsilon^c F_{Y}\paren{\epsilon}=\infty}.$ 

Let $\alpha>\beta>d_2,$ $I_k=\set{y\in Y:2^{-\paren{k+1}}<y\leq 2^{-k}},$ and $C_0=\sum_{\set{\epsilon\in Y:\epsilon>1}}\epsilon^\alpha.$ By definition, there is a constant $C_1>0$ so that $F(\epsilon)< C_1 \epsilon^{-\beta}$ for all $\epsilon\in\R.$  Then
\begin{align*}
\sum_{y\in Y} y^\alpha\leq &\;\; \sum_{k=0}^{\infty} \abs{I_k} 2^{-\alpha k} +C_0\\
< &\;\; \sum_{k=0}^{\infty} F\paren{2^{-\paren{k+1}}} 2^{-\alpha k}+C_0\\
< &\;\; \sum_{k=0}^{\infty} C_1 2^{\beta\paren{k+1}}2^{-\alpha k}+C_0\\
= & \;\; C_1 2^{\beta}\sum_{k=1}^{\infty} 2^{\paren{\beta-\alpha}k}+C_0\\
< & \; \; \infty\,,
\end{align*}
because $\beta-\alpha<0.$ Therefore, $d_1\leq d_2.$

Let $\alpha<d_2$ and $D>0.$ $\lim_{\epsilon\rightarrow 0} \epsilon^{\alpha} F\paren{\epsilon}=\infty$ so there is an $\epsilon_0>0$ so that $\epsilon_0^{\alpha}F_Y\paren{\epsilon_0}>D.$ Then
\[\sum_{y\in Y} y^\alpha\geq \sum_{y>\epsilon_0} y^\alpha>\epsilon_0^{\alpha} F_Y\paren{\epsilon_0}>D\,,\]
so $\sum_{y\in Y} y^\alpha$ is greater than any positive real number and $\sum_{y\in Y} y^\alpha=\infty.$ Therefore $d_1\geq d_2,$ and $d_1=d_2.$
\end{proof}

We also show that $\text{dim}_{\widehat{\PH}}^i$ is bounded above by $\text{dim}_{\PH}^i:$

\begin{Proposition}
Let $X$ be a compact subset of a metric space. Then
\s{\text{dim}_{\widehat{\PH}}^i\paren{X}\leq \text{dim}_{\PH}^i\paren{X}\,.}
\end{Proposition}
\begin{proof}
Let $\alpha<\text{dim}_{\widehat{\PH}}^i\paren{X}$ and $D>0.$ There is a finite collection of intervals $\set{b_j,d_j}_{j=1}^n\subset \PH_i\paren{X}$ with 
\begin{equation}
\label{eq_hat3}
\sum_{j=1}^n\paren{d_j-b_j}^\alpha>D+1\,.
\end{equation}
The function $f\paren{x}=x^\alpha$ is uniformly continuous on $\left[\min_j\paren{d_j-b_j},\max_j\paren{d_j-b_j}\right]$ so there is a $\epsilon>0$ so that 
\begin{equation}
\label{eq_hat2}
\abs{x^\alpha-\paren{x+\delta}^\alpha}<\frac{1}{n}
\end{equation}
for all $x\in\brac{\min_j\paren{d_j-b_j},\max_j\paren{d_j-b_j}}$ and $\abs{\delta}<\epsilon.$ Let $\textbf{x}\subset X$ be a finite point satisfying $d_H\paren{\textbf{x},X}<\epsilon.$ By bottleneck stability, the intervals $\set{\paren{b_j,d_j}}$ can be paired with intervals $\set{\paren{\hat{b}_j,\hat{d}_j}}\in\PH_i\paren{\textbf{x}}$ so that  
\begin{equation}
\label{eq_hat1}
\abs{\paren{d_j-b_j}-\paren{\hat{d}_j-\hat{b}_j}}<\epsilon\,,
\end{equation}
for $j=1,\ldots,n.$ 
(Note that $\textbf{x}$ might have other, smaller intervals, but it doesn't matter for our computations.) Then
\begin{align*}
E^i_{\alpha}\paren{\textbf{x}}=&\;\;\sum_{\paren{\hat{b},\hat{d}}\in\PH_i\paren{\textbf{x}}}\paren{\hat{d}-\hat{b}}^\alpha\\
\geq&\;\; \sum_{j=1}^n\paren{\hat{d}_j-\hat{b}_j}^\alpha\\
\geq&\;\; \sum_{j=1}^n\paren{\paren{d_j-b_j}^\alpha-\frac{1}{n}} &&\text{using Eqns.~\ref{eq_hat2} and~\ref{eq_hat1}}\\
>&\;\;D &&\text{by Eqn.~\ref{eq_hat3}}\,. 
\end{align*}
$D$ was arbitrary so  $E^i_{\alpha}\paren{\textbf{x}}$ can be made arbitrarily large and $\text{dim}_{\PH}^i\paren{X}\geq\alpha$ for all $\alpha<\text{dim}_{\widehat{\PH}}^i\paren{X}$ and $\text{dim}_{\PH}^i\paren{X}\geq \text{dim}_{\widehat{\PH}}^i\paren{X},$ as desired.
\end{proof}

\section{Triangle Computations}
\label{sec:triangleComputations}
Here, we complete the computations in the proof of Proposition~\ref{triangleLemma}.
\subsection{Minimizing $TP_1$}
\label{sec:triangleComputations1}

We will find the minimum of the function
\[TP_1\paren{x,y_1,y_2}=\frac{\sqrt{\paren{x^2 + y_1^2} \paren{x^2 + y_2^2}}}{2 x}-\frac{y_1-y_2}{2}\]
subject to the constraints  $c\sqrt{N}\leq y_1 \leq N,$ $-N \leq y_2 \leq -c\sqrt{N},$ and $c\sqrt{N}+\sqrt{-y_1y_2}\leq x \leq N.$ The partial derivative with respect to $x$ is

\[\frac{\partial TP_1\paren{x,y_1,y_2}}{\partial x}=\frac{ 2x^2+y_1^2+y_2^2}{2 \sqrt{\left(x^2+y_1^2\right) \left(x^2+y_2^2\right)}}-\frac{\sqrt{\left(x^2+y_1^2\right) \left(x^2+y_2^2\right)}}{2 x^2}=\frac{x^4-y_1^2 y_2^2}{2 x^2 \sqrt{\left(x^2+y_1^2\right) \left(x^2+y_2^2\right)}}\,,\]
which is positive when $x>\sqrt{-y_1 y_2}.$ Therefore, the minimum value of $TP_1$ is achieved when $x$ is as small as possible, $x=c\sqrt{N}+\sqrt{-y_1y_2}.$ Also, 
\[\frac{\partial TP_1\paren{x,y_1,y_2}}{\partial y_1}=\frac{y_1 \left(x^2+y_2^2\right)}{2 x \sqrt{\left(x^2+y_1^2\right) \left(x^2+y_2^2\right)}}-\frac{1}{2}=\frac{-x \sqrt{\left(x^2+y_1^2\right) \left(x^2+y_2^2\right)}+x^2 y_1+y_1 y_2^2}{2 x \sqrt{\left(x^2+y_1^2\right) \left(x^2+y_2^2\right)}}\,,\]

and (using that $x,y_1,$ and $-y_2$ are positive)
\begin{align*}
\frac{\partial TP_1\paren{x,y_1,y_2}}{\partial y_1} = 0 \iff & \;\;  -x \sqrt{\left(x^2+y_1^2\right) \left(x^2+y_2^2\right)}+x^2 y_1+y_1 y_2^2 =0 \\
\iff & \;\; \left(x^2 y_1+y_1 y_2^2\right)^2-x^2 \left(x^2+y_1^2\right) \left(x^2+y_2^2\right) = 0 \\
\iff & \;\; \left(x^2+y_2^2\right) \left(x^4-y_1^2 y_2^2\right)=0\\
\iff &\;\; x = \pm\sqrt{-y_1 y_2}\,\,
\end{align*} 
which is also not allowed by our constraints. We can see (for example, by plugging in $x=2 y_1,y_2=-y_1$ above) that this partial derivative is always negative when our constraints hold, and the minimum is achieved when $y_1$ is as large as possible, $y_1=N.$ A symmetric computation shows that $y_2=-N$ at the minimum. Therefore, 
\begin{align*}
TP_1\paren{x,y_1,y_2}\geq &\;\; TP_1\paren{c\sqrt{N}+N,N,-N}\\
= & \;\; \frac{\sqrt{\left(\left(c \sqrt{N}+N\right)^2+N^2\right)^2}}{2 \left(c \sqrt{N}+N\right)}-N\\
= & \;\; \frac{c^2 N}{2 \left(c \sqrt{N}+N\right)}\,.
\end{align*}

\subsection{Minimizing $TP_2$}
\label{sec:triangleComputations2}
We will find the minimum of the function
\[TP_2\paren{x,y_1,y_2}=\frac{\sqrt{\paren{x^2 + y_1^2} \paren{x^2 + y_2^2}}}{2 x}-\frac{\sqrt{x^2+y_1^2}}{2}\]
subject to the constraints  $c\sqrt{N}\leq y_1 \leq N,$ $-N \leq y_2 \leq -c\sqrt{N},$ and $c\sqrt{N}\leq x \leq N.$ The partial derivative of $TP_2$ with respect to $y_1$ is 
\[\frac{\partial TP_2\paren{x,y_1,y_2}}{\partial y_1}=\frac{y_1 \left(x^2+y_2^2\right)}{2 x \sqrt{\left(x^2+y_1^2\right) \left(x^2+y_2^2\right)}}-\frac{y_1}{2 \sqrt{x^2+y_1^2}}=\frac{y_1 \left(\sqrt{x^2+y_2^2}-x\right)}{2 x \sqrt{x^2+y_1^2}}\,,\]
which is always positive. Therefore the minimum will occur when $y_1=c\sqrt{N}.$ A symmetric computation shows that $y_2$ will equal $-c\sqrt{N}$ at the minimum. Plugging these values into $TP_2$ gives
\[TP_2\paren{x,c\sqrt{N},-c\sqrt{N}}= \frac{c^2 N+x^2}{2 x}-\frac{1}{2} \sqrt{c^2 N+x^2}\,.\]
This is a decreasing function of $x.$ To see this, take the derivative
\begin{align*}
\frac{\partial TP_2\paren{x,c\sqrt{N},-c\sqrt{N}}}{\partial x}=&\;\;\frac{1}{2}\frac{\partial\paren{ \sqrt{x^2+2c^2 N + c^4 N^2/x^2}-\sqrt{x^2+c^2 n}}}{\partial x}\\
= &\;\; \frac{1}{2}\paren{\paren{x-\frac{c^4 N^2}{x^3}}\paren{x^2+2c^2 N+\frac{c^4 N^2}{x^2}}^{-\frac{1}{2}}-x\paren{x^2+c^2 N}^{-\frac{1}{2}}}\\
< &\;\; \frac{x}{2}\paren{\paren{x^2+2c^2 N+\frac{c^4 N^2}{x^2}}^{-\frac{1}{2}}-\paren{x^2+c^2N}^{-\frac{1}{2}}}\\
< & \;\; 0\,\,
\end{align*}
because the function $y\mapsto y^{-\frac{1}{2}}$ is decreasing. Therefore, the minimum value of $TP_2$ is attained when $x$ is as large as possible and
\[TP_2\paren{x,y_1,y_2}\geq TP_2\paren{N,c\sqrt{N},-c\sqrt{N}}= \frac{1}{2} \left(c^2+N-\sqrt{N \left(c^2+N\right)}\right)\,,\]
which is a decreasing function of $N$ by an argument similar to the previous one.


\bibliographystyle{spmpsci}

\end{document}